\documentclass[10.9pt,a4paper]{amsart}
\usepackage[english]{babel}
\usepackage{a4wide}

\usepackage{amsfonts, amstext, amsmath, amsthm, amscd, amssymb, amsxtra,enumitem}
\usepackage{tikz}
\usetikzlibrary{calc, matrix, cd, decorations.markings}
\usetikzlibrary{decorations.pathmorphing}
\usetikzlibrary{decorations.pathreplacing}
\usepackage{float, placeins}
\usepackage[all]{xy}
\usepackage{graphicx,color,float}
\usepackage{comment}
\usepackage{eucal, xspace}
\usepackage{epic}
\usepackage{eepic}
\usepackage[
  textwidth=1.2in,
  backgroundcolor=yellow,
  bordercolor=orange,
  textsize=small
]{todonotes}

\def\to{\mathchoice{\longrightarrow}{\rightarrow}{\rightarrow}{\rightarrow}}
\makeatletter
\newcommand{\shortxra}[2][]{\ext@arrow 0359\rightarrowfill@{#1}{#2}}
\def\longrightarrowfill@{\arrowfill@\relbar\relbar\longrightarrow}
\newcommand{\longxra}[2][]{\ext@arrow 0359\longrightarrowfill@{#1}{#2}}
\renewcommand{\xrightarrow}[2][]{\mathchoice{\longxra[#1]{#2}}%
  {\shortxra[#1]{#2}}{\shortxra[#1]{#2}}{\shortxra[#1]{#2}}}
\makeatother

\def\Nopagebreak{\@nobreaktrue\nopagebreak}
\makeatother

\theoremstyle{plain}
\newtheorem{theorem}{Theorem}[section]
\newtheorem{proposition}[theorem]{Proposition}
\newtheorem{corollary}[theorem]{Corollary}
\newtheorem{lemma}[theorem]{Lemma}

\newtheorem*{claim}{Claim}
\newtheorem*{theorem*}{Theorem}
\theoremstyle{definition}
\newtheorem{definition}[theorem]{Definition}
\newtheorem{question}[theorem]{Question}
\newtheorem{example}[theorem]{Example}
\newtheorem{remark}[theorem]{Remark}

\def\Z{\mathbb{Z}}
\def\Q{\mathbb{Q}}

\def\C{\mathbb{C}}

\def\Im{\operatorname{Im}}
\def\Hom{\operatorname{Hom}}

\def\lk{\ell k}

\def\sign{\operatorname{sign}}

\def\ba{\begin{array}}
\def\ea{\end{array}}

\def\setminus{\smallsetminus}
\def\sm{\setminus}

\newcommand{\eps}{\varepsilon}

\let\oldsharp=\# \def\#{\mathbin{\oldsharp}}

\DeclareMathOperator{\im}{im}

\DeclareMathOperator{\cl}{cl}

\def\emptystr{}
\newcommand{\mkc}[2][]{\begin{color}{red}#2%
  \def\tempstr{#1}%
  \ifx\tempstr\emptystr \else\textsf{\SMALL\ \raise.7ex\hbox{[\tempstr]}}\fi
\end{color}}

\newcommand{\bsm}{\left(\begin{smallmatrix}}
\newcommand{\esm}{\end{smallmatrix}\right)}

\newcommand{\Bl}{\operatorname{Bl}}

\providecommand{\im}{\mathop{\rm im}\nolimits}

\providecommand{\Hom}{\mathop{\rm Hom}\nolimits}

\providecommand{\ev}{\mathop{\rm ev}\nolimits}
\providecommand{\PD}{\mathop{\rm PD}\nolimits}

\providecommand{\ev}{\mathop{\rm ev}\nolimits}

\usepackage[colorlinks,
    linkcolor={blue!50!black},
    citecolor={red!50!black},
    urlcolor={blue!80!black}]{hyperref}

\begin{document}
\title{Abelian invariants of doubly slice links}
\author{Anthony Conway}
\address{Massachusetts Institute of Technology, Cambridge MA 02139}
\email{anthonyyconway@gmail.com}

\author{Patrick Orson}
\address{Department of Mathematics, ETH Z\"{u}rich, Switzerland}
\email{patrick.orson@math.ethz.ch}

\begin{abstract}
We provide obstructions to a link in $S^3$ arising as the cross section of any number of unlinked spheres in $S^4$.
Our obstructions arise from the multivariable signature, the Blanchfield form and generalised Seifert matrices.
We also obtain obstructions in the case of surfaces of higher genera, leading to a lower bound on the doubly slice genus of links.
\end{abstract}
\maketitle

\section{Introduction}

A knot
$K\subset S^3$
 is \emph{doubly slice} if it arises as the intersection of a locally flat unknotted~$2$-sphere in~$S^4$ with the equatorial~$S^3$.
The first detailed study of doubly slice knots was made by Sumners~\cite{Sumners}, and the concept arises as an extension of the ideas of Fox~\cite{MR0140099}.
In this article, we study several natural generalisations of this notion to links.
Given an $n$-component oriented link~$L \subset S^3$, and some $1\leq \mu\leq n$, the link is called~\emph{$\mu$-doubly slice} if it arises as the equatorial cross-section of~$\mu$ oriented unlinked $2$-spheres in $S^4$
(each intersecting the equator $S^3$ nontrivially),
inducing the given orientation on~$L$. The cases~$\mu=1$ and~$\mu=n$ have previously been studied in~\cite{Donald,MR4163852,McCoyMcDonald} and are called respectively \emph{weakly} and \emph{strongly} doubly slice.

\begin{figure}[h]
    \begin{center}
    \begin{tikzpicture}[scale=0.9]
    \node[inner sep=0pt] at (-6.5,-0.9)
    {\includegraphics[width=0.3\textwidth]{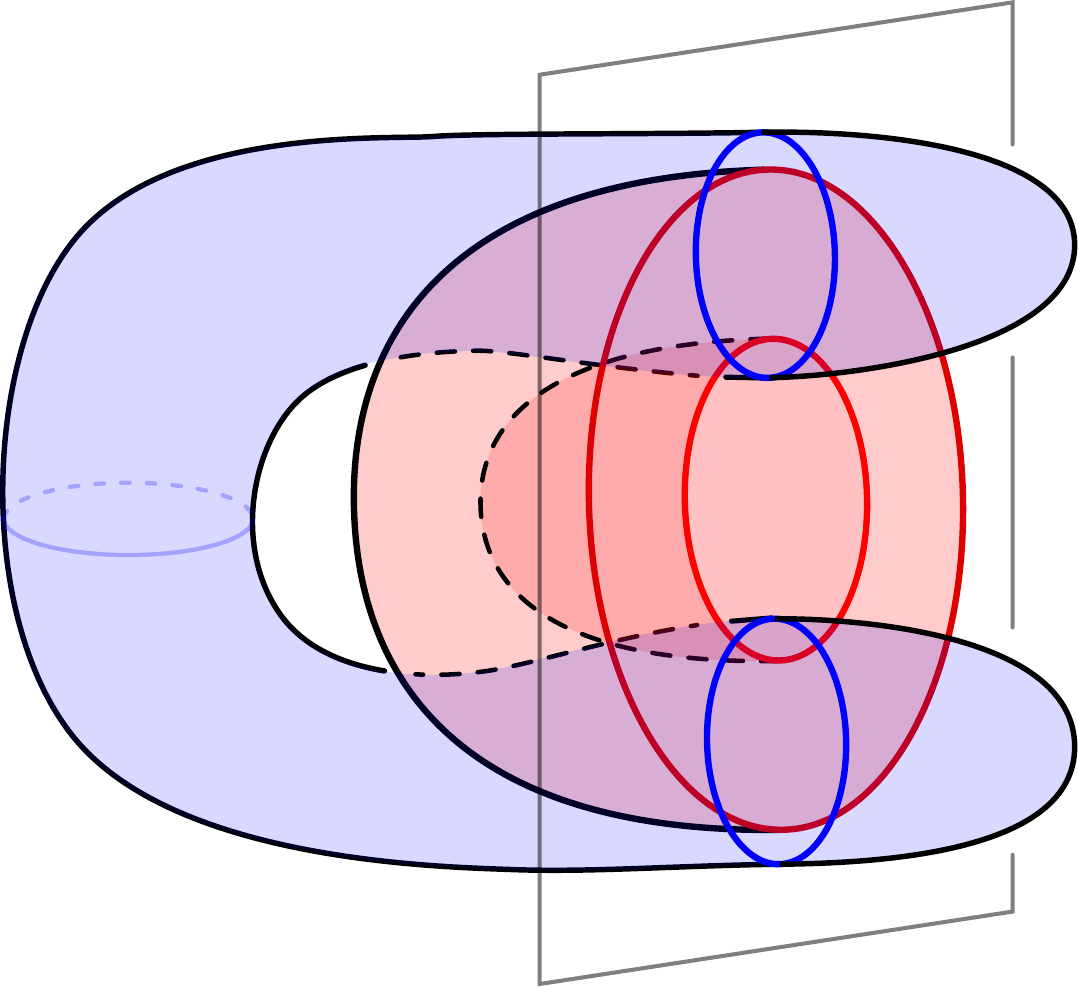}};
    
    \node[inner sep=0pt] at (-9.9,0.5) {\scalebox{1}{$S^4$}};
        \node[inner sep=0pt] at (-4.1,-3) {\scalebox{1}{$S^3$}};
                \node[inner sep=0pt] at (4.5,0) {\scalebox{1}{$\,$}};
                
                \draw [->] [decorate, decoration={snake}, thick] (-3.3,-1) -- (-1.37,-1);

   
\begin{scope}[xscale=-1, shift={(0,0)}]      
   
\begin{scope}[shift={(0,0)}]   
    \draw[thick, blue] (0,0)
to [out=-20, in=up] (0.5,-0.5);
    \draw[line width =1.5mm, white] (0.5,0)
to [out=down, in=up] (0,-0.5);
    \draw[thick, red] (0.5,0)
to [out=down, in=up] (0,-0.5);
\end{scope}

\begin{scope}[shift={(0,-0.5)}]   
    \draw[thick, red] (0,0)
to [out=down, in=up] (0.5,-0.5);
    \draw[line width =1.5mm, white] (0.5,0)
to [out=down, in=up] (0,-0.5);
    \draw[thick, blue] (0.5,0)
to [out=down, in=up] (0,-0.5);
\end{scope}

\begin{scope}[shift={(0,-1)}]   
    \draw[thick, blue] (0,0)
to [out=down, in=up] (0.5,-0.5);
    \draw[line width =1.5mm, white] (0.5,0)
to [out=down, in=up] (0,-0.5);
    \draw[thick, red] (0.5,0)
to [out=down, in=up] (0,-0.5);
\end{scope}

\begin{scope}[shift={(0,-1.5)}]   
    \draw[thick, red] (0,0)
to [out=down, in=up] (0.5,-0.5);
    \draw[line width =1.5mm, white] (0.5,0)
to [out=down, in=20] (0,-0.5);
    \draw[thick, blue] (0.5,0)
to [out=down, in=20] (0,-0.5);
\end{scope}

\end{scope}

   
\begin{scope}[xscale=1, shift={(0.5,0)}]      
   
\begin{scope}[shift={(0,0)}]   
    \draw[thick, blue] (0,0)
to [out=-20, in=up] (0.5,-0.5);
    \draw[line width =1.5mm, white] (0.5,0)
to [out=200, in=up] (0,-0.5);
    \draw[thick, red] (0.5,0)
to [out=200, in=up] (0,-0.5);
\end{scope}

\begin{scope}[shift={(0,-0.5)}]   
    \draw[thick, red] (0,0)
to [out=down, in=up] (0.5,-0.5);
    \draw[line width =1.5mm, white] (0.5,0)
to [out=down, in=up] (0,-0.5);
    \draw[thick, blue] (0.5,0)
to [out=down, in=up] (0,-0.5);
\end{scope}

\begin{scope}[shift={(0,-1)}]   
    \draw[thick, blue] (0,0)
to [out=down, in=up] (0.5,-0.5);
    \draw[line width =1.5mm, white] (0.5,0)
to [out=down, in=up] (0,-0.5);
    \draw[thick, red] (0.5,0)
to [out=down, in=up] (0,-0.5);
\end{scope}

\begin{scope}[shift={(0,-1.5)}]   
    \draw[thick, red] (0,0)
to [out=down, in=160] (0.5,-0.5);
    \draw[line width =1.5mm, white] (0.5,0)
to [out=down, in=20] (0,-0.5);
    \draw[thick, blue] (0.5,0)
to [out=down, in=20] (0,-0.5);
\end{scope}

\end{scope}

   
\begin{scope}[xscale=-1, shift={(-2,0)}]      
   
\begin{scope}[shift={(0,0)}]   
    \draw[thick, blue] (0,0)
to [out=-20, in=up] (0.5,-0.5);
    \draw[line width =1.5mm, white] (0.5,0)
to [out=200, in=up] (0,-0.5);
    \draw[thick, red] (0.5,0)
to [out=200, in=up] (0,-0.5);
\end{scope}

\begin{scope}[shift={(0,-0.5)}]   
    \draw[thick, red] (0,0)
to [out=down, in=up] (0.5,-0.5);
    \draw[line width =1.5mm, white] (0.5,0)
to [out=down, in=up] (0,-0.5);
    \draw[thick, blue] (0.5,0)
to [out=down, in=up] (0,-0.5);
\end{scope}

\begin{scope}[shift={(0,-1)}]   
    \draw[thick, blue] (0,0)
to [out=down, in=up] (0.5,-0.5);
    \draw[line width =1.5mm, white] (0.5,0)
to [out=down, in=up] (0,-0.5);
    \draw[thick, red] (0.5,0)
to [out=down, in=up] (0,-0.5);
\end{scope}

\begin{scope}[shift={(0,-1.5)}]   
    \draw[thick, red] (0,0)
to [out=down, in=160] (0.5,-0.5);
    \draw[line width =1.5mm, white] (0.5,0)
to [out=down, in=20] (0,-0.5);
    \draw[thick, blue] (0.5,0)
to [out=down, in=20] (0,-0.5);
\end{scope}

\end{scope}

   
\begin{scope}[xscale=1, shift={(2.5,0)}]      
   
\begin{scope}[shift={(0,0)}]   
    \draw[thick, blue] (0,0)
to [out=-20, in=up] (0.5,-0.5);
    \draw[line width =1.5mm, white] (0.5,0)
to [out=down, in=up] (0,-0.5);
    \draw[thick, red] (0.5,0)
to [out=down, in=up] (0,-0.5);
\end{scope}

\begin{scope}[shift={(0,-0.5)}]   
    \draw[thick, red] (0,0)
to [out=down, in=up] (0.5,-0.5);
    \draw[line width =1.5mm, white] (0.5,0)
to [out=down, in=up] (0,-0.5);
    \draw[thick, blue] (0.5,0)
to [out=down, in=up] (0,-0.5);
\end{scope}

\begin{scope}[shift={(0,-1)}]   
    \draw[thick, blue] (0,0)
to [out=down, in=up] (0.5,-0.5);
    \draw[line width =1.5mm, white] (0.5,0)
to [out=down, in=up] (0,-0.5);
    \draw[thick, red] (0.5,0)
to [out=down, in=up] (0,-0.5);
\end{scope}

\begin{scope}[shift={(0,-1.5)}]   
    \draw[thick, red] (0,0)
to [out=down, in=up] (0.5,-0.5);
    \draw[line width =1.5mm, white] (0.5,0)
to [out=down, in=20] (0,-0.5);
    \draw[thick, blue] (0.5,0)
to [out=down, in=20] (0,-0.5);
\end{scope}

\end{scope}


\draw[thick, red, looseness=0.38] (-0.5,0)
to [out=up, in=up] (3,0);

\draw[thick, red, looseness=0.5] (1,0)
to [out=20, in=160] (1.5,0);

\draw[thick, blue, looseness=0.5] (0,0)
to [out=20, in=160] (0.5,0);

\draw[thick, blue, looseness=0.5] (2,0)
to [out=20, in=160] (2.5,0);


\draw[thick, red, looseness=0.38] (-0.5,-2)
to [out=down, in=down] (3,-2);

\draw[thick, red, looseness=0.5] (1,-2)
to [out=-20, in=200] (1.5,-2);

\draw[thick, blue, looseness=0.5] (0,-2)
to [out=-20, in=200] (0.5,-2);

\draw[thick, blue, looseness=0.5] (2,-2)
to [out=-20, in=200] (2.5,-2);

    \end{tikzpicture}
        \end{center}
\caption{Unlinked 2-spheres in $S^4$ intersecting $S^3$, and a 2-doubly slice $4$-component link $P(-4,4,-4,4)\subset S^3$ arising as the cross section; cf.~Example~\ref{ex:folding}.}
\label{fig:schematic}
\end{figure}
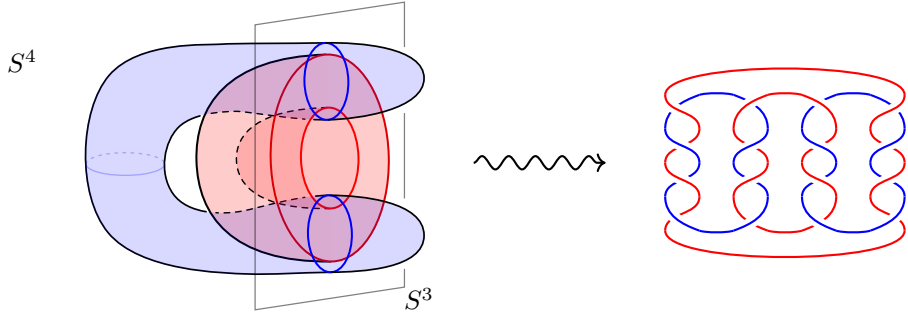

\subsection{Set-up and multivariable abelian invariants}

From now on all links $L\subset S^3$ are assumed to be oriented.
In this paper we will develop several obstructions to a link being $\mu$-doubly slice using the multivariable signature,  the multivariable Blanchfield form and generalised Seifert matrices.
These obstructions are all \emph{multivariable abelian invariants}, so-called because they arise from considering abelian representations of the fundamental group of the link complement.
Such invariants are defined with respect to a choice of \emph{$\mu$-colouring} of an oriented link~$L$ (cf.~\cite{CimasoniFlorens}); that is a choice of an ordered partition of the components into~$\mu$ sublinks~$L=L_1\cup\dots\cup L_\mu$. These $\mu$-colourings naturally arise in the study of doubly slice links, as if~$L$ is $\mu$-doubly slice then each component of~$L$ can be given a colour according to which~$2$-sphere it belongs. Except when $\mu=1$ or $n$, it may be the case that some $\mu$-colourings of a link arise from $\mu$-double slicing where some others may not. Thus for a fixed $\mu$, our task will be to individually rule out all~$\mu$-colourings from arising this way.

From now on all surfaces are assumed to be compact and orientable. As well as studying the doubly-slicing of links by collections of unlinked $2$-spheres, we will also consider an extension of this idea to surfaces of higher genera.
The \emph{doubly slice genus} of a knot is the minimal genus among unknotted surfaces in $S^4$ for which the knot appears as a cross section, and has been studied recently in~\cite{Chen, McDonald, OrsonPowell}. Given an oriented~$n$-component link $L$ and $1\leq \mu\leq n$ we consider the \emph{$\mu$-doubly slice genus}
\begin{equation}\label{eq:dsg}
g^\mu_{ds}(L)=\operatorname{min}\left\lbrace \textstyle{g\, ( \,\mathop{\bigsqcup}_{i=1}^\mu \Sigma_i )} \ \vert\ L_i=\Sigma_i \cap S^3  \text{ for each $i$ and } \textstyle{\mathop{\bigsqcup}_{i=1}^\mu \Sigma_i  \subset S^4} \text{ is unlinked} \right\rbrace,
\end{equation}
where a collection of surfaces in $S^4$ is \emph{unlinked} if it bounds a disjoint collection of handlebodies in $S^4$.
When $\mu=1$, one could call this the \emph{weak} doubly slice genus, and this quantity is always defined. To see this, consider that every oriented link bounds a connected oriented surface in $S^3$ and when such a surface is pushed into $D^4$, then doubled along the boundary, the result is unlinked in $S^4$. 
When $\mu\neq 1$, the set in equation (\ref{eq:dsg}) may be empty, and in that case we define~$g^\mu_{ds}(L):=\infty$.

Besides the abelian invariants that we will describe further below, 
we note there also exist more straightforward tools for studying $\mu$-double sliceness. These will be introduced, as needed, to analyse specific links later in the paper, but we mention some now. For example, a $\mu$-doubly slice link is moreover $(\mu-1)$-doubly slice (simply tube together the~$2$-spheres). 
We also note that all components of a strongly doubly slice link are themselves doubly slice, meaning in particular that the link itself is necessarily slice.
Finally, when $g^\mu_{ds}(L)<\infty$ (in particular, when $L$ is $\mu$-doubly slice), $L_i=\Sigma_i\cap S^3$ for a collection of surfaces $\Sigma_i$ that bound disjoint handlebodies. We may assume the handlebodies meet $S^3$ transversely, and thus conclude $L$ bounds a collection of $\mu$ disjoint surfaces in~$S^3$ (with the $i$-coloured sublink bounding the $i^{\text{th}}$ surface).
For $\mu\neq 1$, this is a fairly restrictive condition and indeed for $\mu=n$ is the well-known notion of a \emph{boundary link}.

We next briefly recall previous results concerning doubly slice knots and the doubly slice genus of knots, for context.

\subsection{Doubly slice knots and doubly slice genus of knots}
Among the early results on doubly slice knots, Sumners~\cite{Sumners} proved that if a knot~$K$ is doubly slice, then there exists a Seifert surface and a basis of curves so that the Seifert matrix is \emph{hyperbolic}:
\[
A=\begin{pmatrix}
0 &B \\ B^T &0
\end{pmatrix}.
\]
A knot with a hyperbolic Seifert matrix is called \emph{algebraically doubly slice}. In fact it is now known that every Seifert matrix for an algebraically doubly slice knot is hyperbolic \cite[Theorem~4.14]{MR3604378}.
For any knot~$K$, the \emph{Levine-Tristram signature}~$\sigma_K^{LT}(\omega)$ at~$\omega\in S^1$ is the signature of the complex Hermitian matrix~$(1-\omega)A+(1-\overline{\omega})A^T$, where~$A$ is any Seifert matrix for the knot (the signature of this matrix is independent of the choice of~$A$). 
If~$K$ is algebraically doubly slice then the Levine-Tristram signature vanishes identically on~$S^1$.
It is also known that~$K$ is algebraic doubly slice if and only if its Blanchfield form is hyperbolic. Recall that the Blanchfield form of a knot~$K$ is a non-singular, sesquilinear, Hermitian pairing
\[
 \Bl_K \colon H_1(X_K;\Z[t^{\pm 1}]) \times H_1(X_K;\Z[t^{\pm 1}]) \to \Q(t)/\Z[t^{\pm 1}],
\]
where~$X_K=S^3\sm \nu K$ denotes the knot exterior.
The form is called \emph{hyperbolic} if there are two modules~$P_1,P_2 \subset H_1(X_K;\Z[t^{\pm 1}])$ such that~$P_1 \oplus P_2=H_1(X_K;\Z[t^{\pm 1}])$ and~$P_i=P_i^\perp$ for~$i=1,2$. A consequence of being hyperbolic is that the rational Alexander module~$H_1(X_K;\Q[t^{\pm 1}])$ splits as~$G \oplus \overline{G}$ for some~$\Q[t^{\pm1}]$-module~$G$; see~\cite{LevineMetabolicHyperbolic, KeartonCobordism, KeartonHermitian}.
Here, we write~$\overline{G}$ for the~$\Q[t^{\pm 1}]$-module whose underlying abelian group is~$G$ but with module structure given by~$p(t) \cdot g=p(t^{-1})g$ for~$g \in G$ and~$p(t) \in \Q[t^{\pm 1}]$.

Beyond algebraic double sliceness, metabelian invariants are also known to obstruct double sliceness~\cite{GilmerLivingston, LivingstonMeier}, as are $L^2$-invariants~\cite{KimNew,ChaKim}.
In the smooth category, obstructions have been obtained using the correction terms from Heegaard-Floer homology~\cite{Meier,KimKim}.

Concerning the doubly slice genus, less is known, although the field seems to be advancing briskly.
In~\cite{OrsonPowell}, the second named author and Mark Powell proved that for every~$\omega \in S^1$
\[
|\sigma_K(\omega)| \leq g_{ds}(K).
\]
Other lower bounds were obtained earlier by W.~Chen using Casson-Gordon invariants~\cite{Chen}, while upper bounds were described by McDonald~\cite{McDonald}.

\subsection{Multivariable signature obstructions} Recall that our strategy to obstruct $\mu$-double sliceness of links will be to obstruct it for each given $\mu$-colouring for the link.

\begin{definition}
\label{def:DoublySlice}
Let~$L=L_1\cup\dots \cup L_\mu$ be an~$n$-component~$\mu$-coloured link.
\begin{enumerate}[leftmargin=*]\setlength\itemsep{0em}
\item The link $L$ is \emph{doubly slice} (with respect to the $\mu$-colouring) if there exists a locally flat ordered oriented ~$\mu$-component unlink~$S=S_1 \sqcup \ldots \sqcup S_\mu \subset S^4$ such that~$L_i=S_i \cap S^3$ for~$i=1,\ldots,\mu$, where~$S^3\subset S^4$ is the standard equator.
\item The \emph{doubly slice genus}~$g_{ds}(L)$ of $L$ (with respect to the $\mu$-colouring) is the minimal genus among ordered unlinked~$\mu$-component locally flat, embedded, closed, oriented surfaces in~$S^4$ for which~$L$ arises as a cross section: 
\[
 g_{ds}(L)=\operatorname{min}\left\lbrace \textstyle{g\, ( \,\mathop{\bigsqcup}_{i=1}^\mu \Sigma_i )} \ \vert\ L_i=\Sigma_i \cap S^3  \text{ for each $i$ and } \textstyle{\mathop{\bigsqcup}_{i=1}^\mu \Sigma_i  \subset S^4} \text{ is unlinked} \right\rbrace.
\]
If~$L$ does not arise in this way we write~$g_{ds}(L)=~\infty$.
\end{enumerate}
 \end{definition}

As mentioned above, the Levine-Tristram signature of a doubly slice knot vanishes and provides a lower bound on the doubly slice genus.
To describe the generalisation of this result to links, we set~$\mathbb{T}^\mu=(S^1)^\mu$
and recall from~\cite{CimasoniFlorens} that the multivariable signature of an oriented~$\mu$-coloured link~$L$ is a function
$$ \sigma_L \colon \mathbb{T}^\mu \to \Z$$
that generalises the Levine-Tristram signature~$\sigma_L^{LT} \colon S^1 \to \Z$; we refer to Remark~\ref{rem:VariableEquals1} for a discussion of the domain of $\sigma_L$.
In fact for~$\mu=1$, the multivariable signature coincides with~$\sigma_L^{LT}$.
The multivariable signature is known to obstruct links from bounding collections of surfaces in~$D^4$~\cite{CimasoniFlorens, ConwayNagelToffoli,Viro} (and in more general 4-manifolds~\cite{ConwayNagel}) and from arising as the cross section of surfaces in $S^4$ for certain \emph{proper} subsets of ${\mathbb{T}^\mu} $~\cite[{Theorem 7.3}]{CimasoniFlorens}.

Our first result{, proved in Theorem~\ref{thm:LowerBound4D},} is the following lower bound on the doubly slice genus of a~$\mu$-coloured link; note that this bound holds \emph{for all} ${\boldsymbol{\omega}}\in {\mathbb{T}^\mu} $.
\begin{theorem}
\label{thm:MainTheoremIntro}
If~$L$ is a~$\mu$-coloured link, then for all~${\boldsymbol{\omega}} \in \mathbb{T}^\mu$ one has 
$$|\sigma_L({\boldsymbol{\omega}} )| \leq g_{ds}(L).$$
In particular, the multivariable signature of a doubly slice coloured link vanishes.
\end{theorem}

Applying Theorem~\ref{thm:MainTheoremIntro} with~$\mu=1$, we obtain:
\begin{corollary}
\label{cor:WeaklyDoublySlice}
If an oriented link~$L$ arises as the cross section of a single unknotted genus~$g$ surface~$\Sigma \subset S^4$, then for all~$\omega \in S^1$, one has~$ |\sigma_L^{LT}(\omega)| \leq g.$
In particular, the Levine-Tristram signature of a weakly doubly slice link vanishes.
\end{corollary}

Corollary~\ref{cor:WeaklyDoublySlice} immediately implies that following result.

\begin{corollary}
\label{cor:StronglyDoublySliceLT}
If an~$n$-component ordered link~$L$ arises as the cross section of an~$n$-component ordered unlinked {genus~$g$} surface~$\Sigma=\Sigma_1 \sqcup \ldots \sqcup \Sigma_n \subset S^4$, then for all~$\omega \in S^1$, one has~$ |\sigma_L^{LT}(\omega)| \leq g.$
In particular, the Levine-Tristram signature of a strongly doubly slice link vanishes.
\end{corollary}

Applying Theorem~\ref{thm:MainTheoremIntro} with~$\mu=n$, we obtain:

\begin{corollary}
\label{cor:StronglyDoublySliceMultivariable}
If an~$n$-component ordered link~$L$ arises as the cross section of an~$n$-component unlinked {genus~$g$} surface~$\Sigma=\Sigma_1 \sqcup \ldots \sqcup \Sigma_n \subset S^4$, then for all~${\boldsymbol{\omega}}  \in \mathbb{T}^n$, one has~$ |\sigma_L({\boldsymbol{\omega}} )| \leq~g.$
In particular, the multivariable signature of a strongly doubly slice link vanishes.
\end{corollary}

Given a $\mu$-coloured link $L$, Cimasoni and Florens also introduced a multivariable nullity function $\eta_L \colon \mathbb{T}_*^\mu  \to \Z$~\cite[Section 2]{CimasoniFlorens}, where $\mathbb{T}_*^\mu:=(S^1\setminus \lbrace 1 \rbrace)^\mu$.
This function is non-zero at $\boldsymbol{\omega} \in \mathbb{T}_*^\mu $ if and only if~$\Delta_L(\boldsymbol{\omega})=0$ and, among other properties, is known 
to obstruct $L$ from arising as a cross section of surfaces in~$S^4$ for certain \emph{proper} subsets of $\mathbb{T}_*^\mu$~\cite[Theorem 7.3]{CimasoniFlorens}.
 Corollary~\ref{cor:MultivariableNullity} shows that \emph{for all}~${\boldsymbol{\omega}}\in \mathbb{T}_*^\mu$ 
the nullity~$\eta_L$ provides an obstruction to a coloured link being doubly slice.

\begin{proposition}
\label{prop:NullityPropIntro}
If a~$\mu$-coloured link $L$
has $g_{ds}(L)<\infty$ (for example, if it is doubly slice)
, then $\eta_L({\boldsymbol{\omega}}) \geq \mu-1$ for all ${\boldsymbol{\omega}}\in \mathbb{T}^\mu_*$.
\end{proposition}

\subsection{The Blanchfield pairing of doubly slice links}\label{sec:introblanchfield}

As we mentioned above, the Blanchfield pairing of a doubly slice knot~$K$ is hyperbolic.
The second goal of this paper is to generalise this result to links. We define:
\[
\Lambda:=\Z[t_1^{\pm 1},\ldots,t_\mu^{\pm 1}], \quad
\Lambda_S:=\Z[t_1^{\pm 1},\ldots,t_\mu^{\pm 1},(1-t_1)^{-1},\ldots,(1-t_\mu)^{-1}], \quad
Q:=\Q(t_1,\ldots,t_\mu),
\]
and denote by~$X_L:=S^3\sm \nu L$ the link exterior, i.e.~the complement of an open tubular neighbourhood of~$L$.
We  also write $H_*(X_L;\Lambda)$ for the homology of the $\Z^\mu$-cover of $X_L$ that arises from the homomorphism $\pi_1(X_L) \to \Z^\mu$; $\gamma \mapsto (\ell k(L_1,\gamma),\ldots,\ell k(L_\mu,\gamma))$.
Deck transformations endow~$H_*(X_L;\Lambda)$ with the structure of a $\Lambda$-module and we set $H_1(X_L;\Lambda_S)=H_1(X_L;\Lambda) \otimes_\Lambda \Lambda_S$, referring to Section~\ref{sec:Twisted} for further discussion of twisted homology.
Before focusing on the torsion submodule of $H_1(X_L;\Lambda_S)$, we note the following fact about its rank (proved in Lemma~\ref{lem:AlexanderModuleSplits}), which is sometimes referred to as the \emph{Alexander nullity} of the coloured link: $\beta(L):=\operatorname{rk}_{\Lambda}H_1(X_L;\Lambda)$.

\begin{proposition}
\label{prop:AlexanderNullity}
If a $\mu$-coloured link $L$
has $g_{ds}(L)<\infty$ (for example, if it is doubly slice)
, then $\beta(L) \geq \mu-1$.
In particular, for~$\mu>1$, the Alexander polynomial of {a} doubly slice $\mu$-coloured link $L$ vanishes identically.
\end{proposition}

 In this article, \emph{the Blanchfield pairing} of a $\mu$-coloured link~$L$ will refer to a sesquilinear Hermitian pairing on the torsion submodule~$TH_1(X_L;\Lambda_S)\subset H_1(X_L;\Lambda_S)$
\[
\Bl_L \colon TH_1(X_L;\Lambda_S) \times TH_1(X_L;\Lambda_S)  \to Q/\Lambda_S,
\]
whose precise definition can be found in Subsection~\ref{sub:BlanchfieldColoured}.
{If $\mu=1$ and if we work over the PID $\Q[t^{\pm 1},(1-t)^{-1}]$ instead of $\Z[t^{\pm 1},(1-t)^{-1}]$, then we refer to the resulting pairing as the (single-variable) \emph{rational Blanchfield pairing of $L$}.}

\begin{remark}It is also possible to define a Blanchfield pairing on~$TH_1(X_L;\Lambda)\subset H_1(X_L;\Lambda)$, but there are several algebraic advantages to working over the ring~$\Lambda_S$, and we will use these in the paper.
Firstly~$H_1(X_L;\Lambda_S)$ always admits a square presentation matrix~\cite[Corollary~3.6]{CimasoniFlorens}, while this is not true for~$H_1(X_L;\Lambda)$~\cite{CrowellStrauss}.
Secondly~$\Bl_L$ can be described using generalised Seifert matrices~\cite{Conway, ConwayFriedlToffoli} while no such result is known over~$\Lambda$. Thirdly, over~$\Lambda_S$ there are criteria under which~$\Bl_L$ is known to be nonsingular \cite{Hillman}.
In particular, if~$L$ is a boundary link, then~$\Bl_L$ is non-singular; see Remark~\ref{rem:BlanchfieldNonSingular} and Proposition~\ref{prop:Nonsingular}.
\end{remark}

We say the Blanchfield form~$\Bl_L$ is \emph{hyperbolic} if there are submodules~$P_1,P_2 \subset TH_1(X_L;\Lambda_S)$ such that~$TH_1(X_L;\Lambda_S)=P_1 \oplus P_2$ and~$P_i=P_i^\perp$ for~$i=1,2$.
What follows is a particular case of {Theorem~\ref{thm:BlanchfieldStronglySlice} which is} our main result  about the Blanchfield pairing of doubly slice coloured links.

\begin{theorem}
\label{thm:BlanchfieldStronglySliceIntro}
Let $L$ be an $n$-component oriented link.
\begin{enumerate}[leftmargin=*]\setlength\itemsep{0em}
\item If $L$ is strongly doubly slice, {then} its $n$-variable Blanchfield pairing is hyperbolic.
\item If $L$ is weakly doubly slice, then its single-variable rational Blanchfield pairing is hyperbolic.
\item If $L$ is weakly doubly slice and also a boundary link, {then} its single-variable Blanchfield pairing is hyperbolic. 
\end{enumerate}
\end{theorem}

Since these conditions are difficult to apply in practice, we also formulate criteria that are easier to verify.
Let $R$ be a Noetherian ring that is a unique factorisation domain. The \emph{order} of an $R$-module~$H$ is the greatest common divisor of the ideal in $R$
generated by all~$m \times m$ minors of an~$m \times n$-presentation matrix for~$H$.
For instance, the (multivariable) Alexander polynomial of a coloured link $L$ is the order of~$H_1(X_L;\Lambda)$. Given $f=f(t_1,\ldots,t_\mu) \in \Q(t_1,\ldots,t_\mu)$, we write $\overline{f}$ for the rational function $f(t_1^{-1},\ldots,t_\mu^{-1})$.
The following result, which is a corollary of the proof of Theorem~\ref{thm:BlanchfieldStronglySliceIntro}, provides a restriction on the torsion submodule $TH_1(X_L;\Lambda_S)$ of the~$\Lambda_S$-Alexander module of a doubly slice link.

\begin{proposition}
\label{prop:ModuleSplit}

The following assertions hold:
\begin{enumerate}[leftmargin=*]\setlength\itemsep{0em}
\item if~$L$ is strongly doubly slice, then there exist submodules~$G_1,G_2 \subset TH_1(X_L;\Lambda_S)$ with~$\operatorname{Ord}(G_1)=\overline{\operatorname{Ord}(G_2)}$ such that
$$ TH_1(X_L;\Lambda_S)=G_1 \oplus G_2;$$
\item if $L$ is weakly doubly slice, then there exists a submodule $G \subset TH_1(X_L;\Q[t^{\pm 1},(1-t)^{-1}])$ such that
$$ TH_1(X_L;\Q[t^{\pm 1},(1-t)^{-1}])=G\oplus {\overline{G};}$$
\item if~$L$ is weakly doubly slice and also a boundary link, then there exist submodules~$G_1,G_2 \subset TH_1(X_L;\Z[t^{\pm 1}])$ with~$\operatorname{Ord}(G_1)=\overline{\operatorname{Ord}(G_2)}$ such that
$$ TH_1(X_L;\Z[t^{\pm 1}])=G_1 \oplus G_2.$$
\end{enumerate}

\end{proposition}

The obstruction given by Proposition~\ref{prop:ModuleSplit} is fairly computable. Indeed, for boundary links~\cite[Theorem 4.6]{Conway} shows that a presentation matrix for $TH_1(X_L;\Lambda_S)$ can be calculated by using generalised Seifert matrices, and we describe this in more detail in Subsections~\ref{sub:SeifertIntro} and~\ref{sub:Ccomplex}.

We speculate that a completely general coloured version of Theorem~\ref{thm:BlanchfieldStronglySliceIntro} might hold if we consider the Blanchfield form as a pairing on $\widehat{t}H_1(X_L;\Lambda_S)$, where given a~$\Lambda_S$-module~$M$, we write~$\widehat{t}M=TM/zM$ for the quotient of~$TM$ by its so-called \emph{maximal pseudo-null submodule}~$zM$; see Remark~\ref{rem:BlanchfieldNonSingular}.
We will not pursue this question here as the resulting coloured obstruction would be unwieldy.

Use $\Sigma_2(L)$ to denote the double branched cover of $S^3$ along $L$.
Arguments similar, but simpler, than those from the proof of Theorem~\ref{thm:BlanchfieldStronglySliceIntro} also give the following result, which appears in Proposition~\ref{prop:doublebranched}.
\begin{proposition}
\label{prop:LinkingFormIntro}
The linking form $\lambda_{\Sigma_2(L)} \colon TH_1(\Sigma_2(L)) \times TH_1(\Sigma_2(L)) \to \Q/\Z$ of a weakly doubly slice link $L $ is hyperbolic, and $TH_1(\Sigma_2(L))=G \oplus G$ for some finite abelian group $G$.
\end{proposition}

\subsection{Generalised Seifert matrices of doubly slice links}
\label{sub:SeifertIntro}

Since doubly slice knots admit hyperbolic Seifert matrices, it is also natural to investigate constraints imposed by double sliceness on (generalisations of) the Seifert matrix.
For simplicity, we only state our results for strongly doubly slice links, referring to Section~\ref{sec:Seifert} for the general coloured case. 

Strongly doubly slice links are boundary links.
Associated to a boundary link~$L=K_1 \cup \ldots \cup K_n$ is a \emph{boundary Seifert matrix}~$A=(A_{ij})_{1 \leq i, j \leq n}$ which consists of~$n^2$ square matrices where~$A_{ii}$ is a Seifert matrix for~$K_i$ and~$A_{ij}^T=A_{ij}$~\cite{KoSeifert}.
Each of these matrices determines a bilinear pairing~$H_i \times H_j \to \Z$, and a boundary Seifert matrix is \emph{hyperbolic} if there are submodules~$G_1^{\pm},\ldots,G_n^{\pm}$ such that~$\operatorname{rk}(G_i^{\pm})=\frac{1}{2}\operatorname{rk}(H_i), H_i=G_i^{-} \oplus G_i^{+}$, and~$A$ vanishes on~$G_i^{\pm} \times G_j^{\pm}$.
When~$n=1$, this recovers the usual notion of a hyperbolic matrix.
A particular case of {Theorem~\ref{thm:HyperbolicSeifert}, which is} our first obstruction to double sliceness in terms of Seifert matrices{,} reads as follows:
\begin{theorem}
\label{thm:SeifertMatrixIntro}
A strongly doubly slice link admits a hyperbolic boundary Seifert matrix.
\end{theorem}

A word of caution: Theorem~\ref{thm:SeifertBlanchfieldIntro} does not assert that \emph{every} boundary Seifert matrix for $L$ is hyperbolic.
A similar caveat holds for Theorem~\ref{thm:SeifertBlanchfieldIntro} below. Next,
we describe how Theorem~\ref{thm:SeifertMatrixIntro} can be used to recover some of our previous obstructions.
By~\cite[Lemma 1]{CimasoniPotential}, every~$\mu$-coloured link~$L$ bounds a $C$-complex, i.e.~a collection of Seifert surfaces~$F=F_1\cup \ldots \cup F_\mu$ that are either disjoint or intersect pairwise along clasp intersections; see Definition~\ref{def:CComplex} for details.
Following~\cite{Cooper,CooperThesis, CimasoniPotential}, given a sequence of signs~$\varepsilon=(\varepsilon_1,\ldots,\varepsilon_\mu)$, there is a \emph{generalised Seifert matrix}~$A^\varepsilon$ that is obtained by computing linking numbers of the form~$\ell k(i^\varepsilon(x),y)$, where~$x,y\in H_1(F)$ and~$i^\varepsilon(x)$ denotes the curve obtained from~$x$ by pushing it off~$F$ in the direction prescribed by~$\varepsilon$.
Summarising, a choice of a $C$-complex $F$ and of a basis for $H_1(F)$ leads to a collection of generalised Seifert matrices $\lbrace A^\varepsilon \rbrace$ for $L$.
One can then combine these~$2^\mu$ matrices to obtain a \emph{$C$-complex matrix}
$$ H(t_1,\ldots,t_\mu)=\sum_\eps\prod_{i=1}^\mu(1-t_i^{\eps_i})\,A^\eps.$$
When~$\mu=1$, a $C$-complex is a Seifert surface,~$A:=A^-$ is a Seifert matrix for the oriented link~$L$ (with~$A^+=A^T$) and a~$C$-complex matrix is~$(1-t^{-1})A+(1-t)A^T$.
Furthermore, generalised Seifert matrices (and the resulting $C$-complex matrices) provide a natural generalisation of the classical Seifert matrix: they can be used to calculate the Alexander module~\cite[Theorem 3.2]{CimasoniFlorens}, the potential function~\cite{CimasoniPotential}, the multivariable signature~\cite[Section 2]{CimasoniFlorens} and the Blanchfield form~\cite{ConwayFriedlToffoli,Conway}.

However, while the multivariable signature, the Alexander polynomial and the Blanchfield form all provide obstructions to sliceness (and now double sliceness), no obstruction in terms of {generalised Seifert matrices or} $C$-complex matrices was known until now.
As a fairly direct consequence of Theorem~\ref{thm:SeifertMatrixIntro}, 
Theorem~\ref{thm:SeifertBlanchfield} establishes a generalisation of the following result:

\begin{theorem}
\label{thm:SeifertBlanchfieldIntro}
Any strongly doubly slice link admits a {collection of  generalised Seifert matrices~$\lbrace A^\varepsilon \rbrace$ where~$A^\varepsilon$ is hyperbolic for each $\varepsilon$, and also admits a} hyperbolic $C$-complex matrix.
\end{theorem}

As we mentioned above, Theorem~\ref{thm:SeifertBlanchfieldIntro} allows us to recover particular cases of Theorems~\ref{thm:MainTheoremIntro} and~\ref{thm:BlanchfieldStronglySlice} (see Corollary~\ref{cor:seifertway}), but the result might also be of independent interest as we now outline.

\begin{remark}
\label{rem:SeifertIsCool}
While Levine's algebraic concordance group~\cite{LevineInvariants} has been generalised to boundary links~\cite{CappellShaneson,KoCS, SheihamThesis,Sheiham2,SheihamRanicki}, in general, link concordance {currently} lacks a meaningful notion of algebraic concordance.
In particular, there are no known restrictions induced by sliceness on generalised Seifert matrices, nor on $C$-complex matrices.
For this reason, it is encouraging that in the doubly slice case such results can be established.
Finally, note that this paper contains two proofs of Theorem~\ref{thm:BlanchfieldStronglySliceIntro}, one more algebraic and one relying on generalised Seifert surfaces.
We hope that these two 
{perspectives} will offer future insights into the possible definitions of algebraic (double) concordance of links.
\end{remark}

\subsection{Examples}
Here are some of the examples where we apply our results.
 Example~\ref{ex:12Butnot34} exhibits an oriented $4$-component link that is weakly doubly slice and $2$-doubly slice, but that is neither $3$-doubly slice nor strongly doubly slice, Example~\ref{ex:L10n36} shows the limitations of our methods by studying a weakly doubly slice link with vanishing torsion submodule of its $\Lambda_S$-Alexander module and vanishing multivariable signature. Proposition~\ref{prop:StronglyExamples} shows that there are no strongly doubly slice links with~$11$ or fewer crossings.
In Section~\ref{sec:weakly} we determine the weakly doubly slice status of all links with 9 or fewer crossings, with the exception of 3 links (with various orientations). This is mainly achieved using our abelian obstructions, but in some cases when our obstructions are ineffective we use more ad hoc arguments based on linking numbers. Of special interest is the case of the Borromean rings, which are not weakly doubly slice, and for which we use an entirely distinct argument that exploits the triple linking.
Finally, Example~\ref{ex:Orientations} exhibits an example of a $2$-component link that is weakly doubly slice with one quasi-orientation, but not with the other, thereby answering a question of McCoy and McDonald~\cite[Question 3]{McCoyMcDonald}.

\subsection*{Organisation}
In Section~\ref{sec:Twisted}, we review some general results we will require about the homology of abelian covers and twisted homology.
Section~\ref{sec:Signature} concerns multivariable signatures and we prove the lower bound on the doubly slice genus from Theorem~\ref{thm:MainTheoremIntro}.
Section~\ref{sec:Blanchfield} is about the Blanchfield pairing, and we prove it is hyperbolic for strongly doubly slice links.
In Section~\ref{sec:Seifert}, we turn to what can be said about various notions of Seifert matrix for a doubly slice link and use these methods to provide alternative proofs for some of the results so far.
In Section~\ref{sec:Examples}, we study examples.

\subsection*{Acknowledgments} The authors would like to thank Chris W.~Davis for helpful conversations and recalling an argument of Peter Teichner for us, which we used in Proposition~\ref{prop:borromean}. We thank
Peter Feller for suggesting Lemma \ref{lem:peterfeller} to us and for helpful conversations.
We thank Duncan McCoy and Clayton McDonald for drawing our attention to~\cite[Question 3]{McCoyMcDonald}. 
We thank the anonymous referee for helpful suggestions.
PO is supported by the SNSF Grant~181199.

\subsection*{Conventions}
Links are assumed to be oriented, unless otherwise stated. Given a ring $R$ with an involution $x \mapsto \overline{x}$ and a~$R$-module~$H$, we write~$\overline{H}$ for the~$R$-module whose underlying abelian group is~$H$ but with module structure given by~$p \cdot h=\overline{p}h$ for~$h \in H$ and~$p \in R$.
From now on, all manifolds are assumed to be compact, based and oriented.
An element ${\boldsymbol{\omega}} \in \mathbb{T}^\mu$ will always have coordinates denoted by ${\boldsymbol{\omega}}=(\omega_1,\ldots,\omega_\mu)$.
We work in the topological category with locally flat embeddings unless otherwise stated, but note this means our results will also hold in the smooth category.

\section{Twisted homology}
\label{sec:Twisted}

In this section, we review some facts about twisted homology.
Here, given a CW-pair~$(X,Y)$, a commutative domain~$R$ with involution and a~$(R,\Z[\pi_1(X)])$-module~$M$, twisted homology and cohomology refers to the~$R$-modules
\begin{align*}
H_*(X,Y;M)&=H_*\left( M \otimes_{\Z[\pi_1(X)]} C_*(\widetilde{X},\widetilde{Y} \right), \\
H^*(X,Y;M)&=H_*\left( \Hom_{\Z[\pi_1(X)]}(\overline{C_*(\widetilde{X},\widetilde{Y})},M)  \right),
\end{align*}
where~$p \colon \widetilde{X} \to X$ denotes the universal cover and~$\widetilde{Y}=p^{-1}(Y)$.
Throughout this article, we assume some familiarity with twisted homology, but refer to~\cite{KirkLivingston} for a general reference, or to~\cite[Chapter 5]{ConwayThesis} for a reference whose conventions match ours.
Subsection~\ref{sub:Lambda} collects facts about twisted homology with coefficients in~$M=R=\Z[t_1^{\pm 1},\ldots,t_\mu^{\pm 1},(1-t_1)^{-1},\ldots,(1-t_\mu)^{-1}]$, while Subsection~\ref{sub:Comega} is concerned with a coefficient system over~$R=\C$.

\subsection{Homology of free abelian covers}
\label{sub:Lambda}
Use~$\Lambda:=\Z[\Z^\mu]=\Z[t_1^{\pm 1},\ldots,t_\mu^{\pm 1}]$ to denote the ring of Laurent polynomials in~$\mu$ variables.
Let~$\Lambda_S$ be the ring obtained from~$\Lambda$ by localising the multiplicative set generated by the~$1-t_i$. 
Let~$(X,Y)$ be a CW pair with a map~$H_1(X) \to \Z^\mu$.
We can then consider the twisted homology modules~$H_i(X,Y;\Lambda)$ and~$H_i(X,Y;\Lambda_S)$.

\begin{remark}
\label{rem:TwistedEasier}
For~$R=\Lambda,\Lambda_S$, the homology~$R$-module~$H_*(X,Y;R)$ can equivalently be described as the homology of the chain complex~$R \otimes_\Lambda C_*(\widehat{X},\widehat{Y})$, where~$p \colon \widehat{X} \to X$ is the~$\Z^\mu$-cover of~$X$ and~$\widehat{Y}=p^{-1}(Y)$.
Since $\Lambda_S$ is flat over $\Lambda$, note also that $H_*(X,Y;\Lambda_S)=\Lambda_S \otimes_\Lambda H_*(X,Y;\Lambda)$.
\end{remark}

We recall the description of the~$0$-th homology module of a space~$X$ as well as the main simplification afforded by the use of~$\Lambda_S$ coefficients.

\begin{lemma}
\label{lem:H0}
Let~$X$ be a connected CW complex, and let~$\varphi \colon \pi_1(X) \to \Z^\mu$ be a homomorphism.
\begin{enumerate}[leftmargin=*]\setlength\itemsep{0em}
\item If~$\varphi$ is surjective, then~$H_0(X;\Lambda)=\Z$.
\item If~$z \in \im(\varphi)$ is non-trivial, then~$H_0(X;\Z[\Z^\mu][(z-1)^{-1}])=0$.
\item If~$\psi \colon \pi_1(X \times S^1) \to \Z^\mu$ is a homomorphism such that~$\pi_1(S^1) \to \pi_1(X \times S^1) \xrightarrow{\psi} \Z^\mu$ sends a generator to a non-trivial element~$z$ of~$\Z^\mu$, then~$H_*(X\times S^1;\Z[\Z^\mu][(z-1)^{-1}])=0$.
\end{enumerate}
\end{lemma}
\begin{proof}
The first two assertions follow from the usual computation of the~$0$-the twisted homology group~\cite[Chapter VI.3]{HiltonStammbach}.
The third statement is in~\cite[Lemma 2.2]{ConwayFriedlToffoli}.
\end{proof}

Let~$X$ be a CW complex, and let~$\varphi \colon \pi_1(X) \to \Z^\mu$ be a homomorphism.
In what follows, we will frequently refer to~$H_1(X;\Lambda)$ as the \emph{Alexander module of~$X$} and to~$H_1(X;\Lambda_S)$ as the \emph{$\Lambda_S$-Alexander module of~$X$}.
The next lemma describes the $\Lambda_S$-Alexander module of a space with free fundamental group; for instance the exterior of an unlinked surface in~$S^4$.
Here and in the sequel, $F^\mu$ denotes the free group on $\mu$ generators.

\begin{lemma}
\label{lem:WedgeOfCirclesLambdaS}
If~$X$ is a CW complex with~$\pi_1(X)={F^\mu}$, then~$H_1(X;\Lambda_S)$ is free of rank~$\mu-1$.
\end{lemma}
\begin{proof} 
The first twisted homology group of a space~$X$ only depends on~$\pi_1(X)$.
Since~$\pi_1(X)={F^\mu}$, we may therefore work with~$Y=\vee_{i=1}^{\mu} S^1$ instead of~$X$.
Use~$\widehat{Y}$ to denote the free abelian cover of~$Y$, so that~$H_1(Y;\Lambda_S)= \Lambda_S \otimes_\Lambda H_1(\widehat{Y})$; recall Remark~\ref{rem:TwistedEasier}.
Since~$Y$ is a graph, we have~$H_1(\widehat{Y})=\ker(\partial)$, where~$\partial$ is the boundary map of the chain complex
$$ 0 \to C_1(\widehat{Y}) \xrightarrow{\partial} C_0(\widehat{Y}) \to 0.$$
Endow~$Y$ with the cell structure with a single~$0$-cell~$z$ and~$\mu$~$1$-cells~$x_1,\ldots,x_\mu$.
Fix a lift~$\widetilde{z}$ of~$z$ to~$\widehat{Y}$ and lifts~$\widetilde{x}_i$ of the~$x_i$ with~$\widetilde{z}$ as a start point.
It follows that~$C_1(\widehat{Y})=\oplus_{i=1}^\mu \Lambda \widetilde{x}_i$, that~$C_0(\widehat{Y})=\Lambda \widetilde{z}$, and that the boundary map is given by~$\partial \widetilde{x}_i=(t_i-1)\widetilde{z}$.
Some linear algebra over~$\Lambda_S$ shows that a basis for~$\Lambda_S \otimes_\Lambda H_1(\widehat{Y})$ is given by~$(1-t_1)\widetilde{x}_2-(1-t_2)\widetilde{x}_1,\ldots,(1-t_{\mu-1})\widetilde{x}_\mu-(1-t_\mu)\widetilde{x}_{\mu-1}$.
This establishes that~$H_1(X;\Lambda_S)=H_1(Y;\Lambda_S)$ is free of rank~$\mu-1$.
\end{proof}

We now focus on link exteriors.
Given a~$\mu$-coloured link~$L$ with exterior~$X_L:=S^3 \setminus \nu L$, we consider the map~$\pi_1(X_L) \to \Z^\mu,\gamma \mapsto (\ell k(L_1,\gamma),\ldots,\ell k(L_\mu,\gamma))$.
The \emph{Alexander module of~$L$} is then the~$\Lambda$-module~$H_1(X_L;\Lambda)$.

\begin{remark}
\label{rem:NullityAtMost}
By applying~\cite[Proposition 2.11]{CochranOrrTeichner}, we deduce that for any $n$-component $\mu$-coloured link $L$, we have $\operatorname{rk}_\Lambda H_1(X_L;\Lambda) \leq n-1$.
\end{remark}

The next lemma is stated (for ordered links) in~\cite[Subsection~2.7]{Hillman}, but no proof is given so we offer one here. 
We remark that an alternative proof with a different flavour is implicit in~\cite[proof of Proposition~3.5.4]{ConwayThesis}.

\begin{lemma}
\label{lem:AlexanderModuleSplits}
Suppose~$L$ is a link that bounds an ordered collection~$F=F_1\cup\dots\cup F_\mu$ of disjoint oriented surfaces.
Then with respect to the associated~$\mu$-colouring,~$\Lambda_S^{\mu-1}$ is a direct summand of~$H_1(X_L;\Lambda_S)$. 
 When~$\mu=n$, it is moreover true that~$H_1(X_L;\Lambda_S)=TH_1(X_L;\Lambda_S) \oplus \Lambda_S^{n-1}$.
\end{lemma}

\begin{proof}
By tubing together components of the same colour, we can and will assume that $F_i$ is connected for all $i=1,\dots,\mu$. Thicken $F$ to $F\times[-1,1]$. The $\Z^\mu$-cover $p\colon \widehat{X}_L\to X_L$ associated to the $\mu$-colouring may be described as follows.

We abuse notation and write the intersection~$F\cap X_L$ from now on as~$F$. The space~$\widehat{X}_L\sm p^{-1}(F)$ is a disjoint union~$\bigsqcup_{a\in\Z^\mu}Y^a$, where~$p\colon Y^a\to X_L\sm F\times[-1,1]$ is a homeomorphism for all~$a\in\Z^\mu$.
Note that
\[
\partial\left( \cl(Y^a)\right)\cong \left(F\times\{-1,1\}\right)\,\cup\,\left(\partial X_L\sm F\times[-1,1]\right).
\]
Write~$\nu Y^a:=Y^a\cup F\times\left(\{-1\}\times(-\varepsilon,0]\right)\cup \left(F\times\{1\}\times[0,\varepsilon)\right)$, for an open collar in~$X_L$ along the~$F\times\{-1,1\}$ part of the boundary. 
A Mayer-Vietoris decomposition for the~$\Z^\mu$-cover~$\widehat{X}_L$ is then given by
\[
\bigsqcup_{a\in\Z^\mu}\bigsqcup_{i=1}^\mu F^a_i\times(-\varepsilon,\varepsilon)\xrightarrow{} \bigsqcup_{a\in\Z^\mu} \nu Y^a\to \widehat{X}_L.
\]
Writing $Y:=X_L\sm F\times[-1,1]$, the associated long exact sequence of~$\Lambda$-modules may be written~as
\[
\dots\xrightarrow{}
H_1(\widehat{X}_L)
\xrightarrow{\partial} H_0(F) \otimes_\Z \Lambda \xrightarrow{\varphi}  H_0(Y) \otimes_\Z \Lambda  \to \dots
\]
With respect to the obvious bases,~$\varphi$ is given by the matrix~$(1-t_1\,\,\,\,1-t_2\,\,\,\,\cdots\,\,\,\,1-t_\mu)$. Passing to the ring~$\Lambda_S$, this matrix is seen to have kernel the free module~$\Lambda_S^{\mu-1}$, similarly to the proof of Lemma \ref{lem:WedgeOfCirclesLambdaS}. As this kernel is the image of~$\partial$, the claimed result follows because free modules are projective and so~$\partial$ splits.

In the case that~$\mu=n$, write the decomposition so far achieved as~$H_1(X_L;\Lambda_S)\cong T\oplus \Lambda_S^{n-1}$. As the rank of~$H_1(X_L;\Lambda_S)$ is at most~$n-1$ for any $n$-component link (by Remark~\ref{rem:NullityAtMost}), the rank of~$T$ must be 0 and so~$T$ is torsion. It follows that~$T\cong TH_1(X_L;\Lambda_S)$.
\end{proof}

The next corollary shows that the splitting of Lemma~\ref{lem:AlexanderModuleSplits} into free and torsion parts holds for all colourings of a boundary link.
\begin{corollary}
\label{cor:SplitModuleColored}
If an $n$-component $\mu$-coloured link $L$ is also a boundary link, then 
\[
H_1(X_L;\Lambda_S)=TH_1(X_L;\Lambda_S) \oplus \Lambda_S^{n-1}
\]
\end{corollary}
\begin{proof}
In this proof, we temporarily set $\Lambda_{n,S}=\Z[t_1^{\pm 1},\ldots,t_n^{\pm 1},(1-t_1)^{-1},\ldots, (1-t_n)^{-1}]$.
By Proposition~\ref{prop:ModuleSplit}, we have the decomposition $H_1(X_L;\Lambda_{n,S})=TH_1(X_L;\Lambda_{n,S}) \oplus \Lambda_{n,S}^{n-1}$.
Tensoring the result with $\Lambda_S$ we obtain
\begin{equation}
\label{eq:Fromntomu}
\Lambda_S \otimes_{\Lambda_{n,S}} H_1(X_L;\Lambda_{n,S})  =\left( \Lambda_S \otimes_{\Lambda_{n,S}} TH_1(X_L;\Lambda_{n,S})  \right) \oplus \Lambda_{S}^{n-1}.
\end{equation}
We apply the universal coefficient spectral sequence with~$E^2_{p,q}=\operatorname{Tor}_p^{\Lambda_{n,S}}(H_q(X_L;\Lambda_{n,S}),\Lambda_S)$ and which converges to $H_{*}(X_L;\Lambda_S)$.
Since~$H_0(X_L;\Lambda_{n,S})=0$ by Lemma~\ref{lem:H0}, a computation using this spectral sequence shows that $H_1(X_L;\Lambda_S)=\Lambda_S \otimes_{\Lambda_{n,S}} H_1(X_L;\Lambda_{n,S})$.
Using Remark~\ref{rem:NullityAtMost}, it follows that both sides of~\eqref{eq:Fromntomu} have rank $n-1$.
In turn, this implies that $\Lambda_S \otimes_{\Lambda_{n,S}} TH_1(X_L;\Lambda_{n,S}) $ coincides with the torsion submodule of $H_1(X_L;\Lambda_S)$.
\end{proof}

\subsection{Homology with~$\C^{\boldsymbol{\omega}}$ coefficients}
\label{sub:Comega}
Given~${\boldsymbol{\omega}} \in \mathbb{T}_*^\mu:=(S^1 \setminus \lbrace 1 \rbrace)^\mu$, the map~$\Z^\mu \to \C,t_i \mapsto \omega_i$ endows~$\C$ with the structure of a~$\Lambda$-module, which we write~$\C^{\boldsymbol{\omega}}$ for emphasis.
Note that~$\C^{\boldsymbol{\omega}}$ is a module over~$\Lambda_S$; indeed none of the coordinates of~${\boldsymbol{\omega}}$ is equal to one.
Let~$(X,Y)$ be a~CW-pair with a homomorphism~$H_1(X) \to \Z^\mu$. 
Composing the induced map~$\Z[\pi_1(X)] \to \Lambda$ with the map~$\Lambda \to \C$ which evaluates~$t_i$ at~$\omega_i$ produces a morphism of rings with involutions.
We write~$H_i(X,Y;\C^{\boldsymbol{\omega}})$ for the resulting twisted homology~$\C$-vector spaces.
As in Remark~\ref{rem:TwistedEasier}, these vector spaces can equivalently be described via $\Z^\mu$-covers as the homology of $\C^{\boldsymbol{\omega}} \otimes_\Lambda C_*(\widehat{X},\widehat{Y})$.

\begin{lemma}
\label{lem:ComegaHomology}
Let~$(X,Y)$ be a CW pair, let~$\varphi \colon H_1(X) \to \Z^\mu=\langle t_1,\ldots,t_\mu \rangle$ be a homomorphism, and let~${\boldsymbol{\omega}} \in \mathbb{T}_*^\mu.$
The following assertions holds: 
\begin{enumerate}[leftmargin=*]\setlength\itemsep{0em}
\item if~$X$ is connected and at least one generator~$t_i$ is in the image of~$\varphi$, then 
\begin{align*}
H_0(X;\C^{\boldsymbol{\omega}})&=0, \\
H_1(X;\C^{\boldsymbol{\omega}})&=\C^{\boldsymbol{\omega}} \otimes_{\Lambda_S} H_1(X;\Lambda_S);
\end{align*}
\item evaluation produces an isomorphism~$H^k(X,Y;\C^{\boldsymbol{\omega}}) \cong \overline{\Hom_\C(H_k(X,Y;\C^{\boldsymbol{\omega}}),\C)}$.
\end{enumerate}
\end{lemma}
\begin{proof}
The proofs can be found in~\cite[Lemmas 2.3 and 2.6]{ConwayNagelToffoli}.
\end{proof}

The next lemma focuses on spaces with free fundamental group.

\begin{lemma}
\label{lem:HomologicalAlgebra}
Let~$X$ be a CW complex with~$\pi_1(X)={F^\mu}$.
For~${\boldsymbol{\omega}} \in \mathbb{T}_*^\mu$, we have
$$ H_1(X;\C^{\boldsymbol{\omega}})=\C^{\mu-1}.$$
\end{lemma}
\begin{proof}
We saw in Lemma~\ref{lem:ComegaHomology} that~$ H_1(X;\C^{\boldsymbol{\omega}})=\C^{\boldsymbol{\omega}} \otimes_{\Lambda_S} H_1(X;\Lambda_S),$ while Lemma~\ref{lem:WedgeOfCirclesLambdaS} established that~$H_1(X;\Lambda_S)=\Lambda_S^{\mu-1}$. 
The result follows by combining these two facts.
\end{proof}

We conclude this section by studying the~$\C^{\boldsymbol{\omega}}$-twisted homology of link exteriors in~$S^3$ and surface exteriors in~$S^4$.
As in Subsection, \ref{sub:Lambda}, given a~$\mu$-coloured link~$L \subset S^3$, we consider the map~$\pi_1(X_L) \to \Z^\mu, \gamma \mapsto (\ell k(L_1,\gamma),\ldots,\ell k(L_\mu,\gamma))$.
After fixing~${\boldsymbol{\omega}} \in \mathbb{T}_*^\mu$, we therefore obtain the~$\C$-vector space~$H_1(X_L;\C^{\boldsymbol{\omega}})$.
For later use, we record the following fact about boundary links.

\begin{proposition}
\label{prop:TwistedHomologyBoundaryLink}

Fix $1\leq \mu \leq n$ and ${\boldsymbol{\omega}} \in \mathbb{T}_*^\mu$.
If~$L$ is an $n$-component link that bounds a ordered collection~$F=F_1\cup\dots\cup F_\mu$ of disjoint oriented surfaces, then
\begin{enumerate}[leftmargin=*]\setlength\itemsep{0em}
\item $b_1^{{\boldsymbol{\omega}}}(X_L) \geq \mu-1$;
\item if $L$ is a boundary link, $b_1^{{\boldsymbol{\omega}}}(X_L) \geq n-1$ with equality if $\Delta_L^{tor}({\boldsymbol{\omega}}) \neq 0$.
\end{enumerate}

\end{proposition}
\begin{proof}
We prove the first item.
Lemma~\ref{lem:AlexanderModuleSplits} implies that $H_1(X_L;\Lambda_S)=M \oplus \Lambda_S^{\mu-1}$ for some $\Lambda_S$-module~$M$.
The result now follows from an application of Lemma~\ref{lem:ComegaHomology}:
$$H_1(X_L;\C^{\boldsymbol{\omega}})=\C^{\boldsymbol{\omega}} \otimes_{\Lambda_S} H_1(X_L;\Lambda_S)  =\left( \C^{\boldsymbol{\omega}} \otimes_{\Lambda_S} M \right)  \oplus  \C^{\mu-1}.$$
We now assume that $L$ is a boundary link and prove the second assertion.
{We temporarily set $\Lambda_n:=\Z[\Z^n]$ and $\Lambda_{n,S}$ for the corresponding localised ring. 
View~$L$ as an~$n$-component ordered link.}
As~$L$ is a boundary link, Lemma~\ref{lem:AlexanderModuleSplits} implies that~$H_1(X_L;\Lambda_{n,S})=TH_1(X_L;\Lambda_{n,S}) \oplus \Lambda_{n,S}^{n-1}$.
For~${\boldsymbol{\omega}} \in \mathbb{T}_*^\mu$, the~$\Lambda_S$-module~$\C^{\boldsymbol{\omega}}$ is also a module over~$\Lambda_{n,S}$.
An application of Lemma~\ref{lem:ComegaHomology} now gives the isomorphisms
$$H_1(X_L;\C^{\boldsymbol{\omega}})=\C^{\boldsymbol{\omega}} \otimes_{\Lambda_{n,S}} H_1(X_L;\Lambda_{n,S})  =\left( \C^{\boldsymbol{\omega}} \otimes_{\Lambda_{n,S}} TH_1(X_L;\Lambda_{n,S}) \right)  \oplus  \C^{n-1}.$$
This shows that $b_1^{\boldsymbol{\omega}}(X_L) \geq n-1$.
To prove the last statement, we must show that if $\Delta_L^{tor}=\operatorname{Ord}(TH_1(X_L;\Lambda_{n,S})) $ does not vanish at~${\boldsymbol{\omega}}$, then $\C^{\boldsymbol{\omega}} \otimes_{\Lambda_{n,S}} TH_1(X_L;\Lambda_{n,S})=0$.
Since $L$ is a boundary link, $TH_1(X_L;\Lambda_{n,S}) $ admits a square presentation matrix $A$; see e.g.~\cite[proof of Theorem 4.6]{Conway}.
In particular, $\det(A)=\Delta_L^{tor}$.
Evaluating the coordinates of this matrix at ${\boldsymbol{\omega}}$, we obtain a square presentation matrix $A({\boldsymbol{\omega}})$ for $\C^{\boldsymbol{\omega}} \otimes_{\Lambda_{n,S}} TH_1(X_L;\Lambda_{n,S})$.
This vector space vanishes if and only if $\Delta_L^{tor}({\boldsymbol{\omega}})=\det(A({\boldsymbol{\omega}})) \neq 0$ and this concludes the proof of the proposition.
\end{proof}

Given a $\mu$-coloured link $L $ and ${\boldsymbol{\omega}} \in \mathbb{T}^\mu_*$, the integer $\eta_L(\boldsymbol{\omega}):=b_1^{\boldsymbol{\omega}}(X_L)$ is also known as the \emph{multivariable nullity} of $L$.
As we recall in Proposition~\ref{prop:CComplexMatrix}, this quantity can also be defined using $C$-complexes.

\begin{corollary}
\label{cor:MultivariableNullity}
If a $\mu$-coloured link $L$
has $g_{ds}(L)<\infty$ (for example, if it is doubly slice)
, then $\eta_L({\boldsymbol{\omega}}) \geq \mu-1$ for all ${\boldsymbol{\omega}}\in \mathbb{T}^\mu_*$ and,
in particular, $\eta_L(\omega) \geq \mu-1$ for all $\omega \in S^1_*$.
\end{corollary}
\begin{proof}
As observed in the introduction, when $g_{ds}(L)<\infty$, the link $L$
bounds an ordered collection~$F=F_1\cup\dots\cup F_\mu$ of disjoint oriented surfaces in $S^3$.
The first assertion now follows from Proposition~\ref{prop:TwistedHomologyBoundaryLink}, while the second is immediate since $\eta_L(\omega)=\eta_L(\omega,\ldots,\omega)$ for all $\omega \in S^1_*$~\cite[Proposition 2.5]{CimasoniFlorens}.
\end{proof}

\begin{remark}
Similarly to our comment immediately before Lemma~\ref{lem:AlexanderModuleSplits}, we remark that Corollary~\ref{cor:MultivariableNullity} admits an alternative proof in terms of $C$-complexes, and this is outlined in~\cite[proof of Proposition~3.5.4]{ConwayThesis}.
\end{remark}

Finally, we record a fact about unlinked surfaces in the~$4$-sphere.

\begin{proposition}
\label{prop:HomologyUnlinkedSurface}
If~$\Sigma \subset S^4$ is a~$\mu$-component unlinked surface of genus~$g$, then its Euler characteristic is~$\chi(X_\Sigma)=2(1-\mu)+2g$, and~$b_1^{\boldsymbol{\omega}}(X_\Sigma)=\mu-1$ for $\boldsymbol{\omega} \in \mathbb{T}_*^\mu$.
\end{proposition}
\begin{proof}
An Euler characteristic argument applied to the decomposition~$S^4=X_\Sigma \cup_{\partial} \nu \Sigma$ shows that~$\chi(X_\Sigma)=2-\chi(\Sigma)=2-2\mu+2g=2(1-\mu)+2g$.
This establishes the first assertion.
For the second assertion, note that since~$\Sigma$ is unlinked one has~$\pi_1(X_\Sigma)={F^\mu}$.
By Lemma~\ref{lem:HomologicalAlgebra}, it follows that~$b_1^{\boldsymbol{\omega}}(X_\Sigma)=\mu-1$, concluding the proof of the proposition.
\end{proof}

\section{The multivariable signature and the doubly slice genus}
\label{sec:Signature}

The goal of this section is to prove Theorem~\ref{thm:MainTheoremIntro} from the introduction which states that the multivariable signature of a coloured link provides a lower bound on its doubly slice genus.
To achieve this, Subsection~\ref{sub:MultivariableSignature} first recalls a definition of the multivariable signature, while Subsection~\ref{sub:ProofLowerBound} is concerned with the proof of Theorem~\ref{thm:MainTheoremIntro}.

\subsection{The multivariable signature}
\label{sub:MultivariableSignature}
Set~$\mathbb{T}_*^\mu=(S^1\setminus \lbrace 1 \rbrace)^\mu$.
We recall a four dimensional definition of the multivariable signature~$\sigma_L \colon \mathbb{T}_*^\mu \to \Z$ of a~$\mu$-coloured link~$L$; further background on this definition of~$\sigma_L$ can be found in~\cite[Section 3]{ConwayNagelToffoli}.
Then, we establish an upper bound on~$\sigma_L({\boldsymbol{\omega}})$ that will be useful to establish the lower bound on~$g_{ds}(L)$.

\begin{definition}
\label{def:colouredBoundingSurface}
A \emph{coloured bounding surface} for a~$\mu$-coloured link~$L$  is a union~$F=F_1 \cup \ldots \cup F_\mu$ of locally flat, embedded, oriented surfaces~$F_i \subset D^4$ with~$\partial F_i=L_i$ and that intersect  transversally in at worst double points.
\end{definition}

The exterior of a coloured bounding surface~$F\subset D^4$ will always be written as~$W_F$.
The homology group~$H_1(W_F)$ is freely generated by the meridians of the components~$F_i$~\cite[Lemma 3.1]{ConwayNagelToffoli}.
Summing up the meridians of the same colour gives rise to an epimorphism~$H_1(W_F) \twoheadrightarrow \Z^\mu$ which extends the epimorphism~$H_1(X_L) \twoheadrightarrow \Z^\mu$ described in Section~\ref{sec:Twisted}.
We therefore obtain twisted homology modules~$H_*(W_F;\Lambda)$ and~$H_*(W_F;\Lambda_S)$.
Additionally, given~${\boldsymbol{\omega}} \in \mathbb{T}^\mu_*$, we obtain a~$\C$-valued twisted-coefficient intersection form~$\lambda_{W_F,\C^{\boldsymbol{\omega}}}$ on~$H_2(W_F;\C^{\boldsymbol{\omega}})$.
The signature~$\sigma_{\boldsymbol{\omega}}(W_F):=\operatorname{sign}(\lambda_{W_F,\C^{\boldsymbol{\omega}}})$ depends only on the coloured link~$L$ \cite[Theorem 4.6]{DegtyarevFlorensLecuona}, justifying the following definition.

\begin{definition}
\label{def:MultivariableSignature1}
The \emph{multivariable signature} of a~$\mu$-coloured link~$L$ at ~${\boldsymbol{\omega}} \in \mathbb{T}_*^\mu$ is defined as
$$\sigma_L({\boldsymbol{\omega}}):=\sigma_{\boldsymbol{\omega}}(W_F),$$
where~$F \subset D^4$ is any coloured bounding surface for~$L$.
\end{definition}

{In its 3-dimensional definition, which we recall in Section~\ref{sec:Seifert}, $\sigma_L(\boldsymbol{\omega})$ is obtained as the signature of a matrix $H(\boldsymbol{\omega})$ that vanishes when one of the coordinates of $\boldsymbol{\omega}$ is equal to $1$.
This explains both why we restrict ourselves to $\mathbb{T}_*^\mu=(S^1 \setminus \lbrace 1 \rbrace)^\mu$ and why the results in the introduction were stated on the whole of $\mathbb{T}^\mu=(S^1)^\mu$.
}
We establish some bounds for the absolute value of this signature.

\begin{proposition}
\label{prop:SignatureBound}
Let~${\boldsymbol{\omega}}\in \mathbb{T}_*^\mu$, and let~$F \subset D^4$ be a coloured bounding surface for a~$\mu$-coloured link~$L$.
The following assertions hold:
\begin{enumerate}[leftmargin=*]\setlength\itemsep{0em}
\item for every~$k \geq 0$, the inclusion induces isomorphisms
\begin{align*}
H_*(X_L;\Lambda_S) &\cong H_*(\partial W_F;\Lambda_S), \\
H_*(X_L;\C^{\boldsymbol{\omega}}) &\cong H_*(\partial W_F;\C^{\boldsymbol{\omega}});
\end{align*}
\item the following inequality holds:
$$ |\sigma_L({\boldsymbol{\omega}})| \leq b_2^{\boldsymbol{\omega}}(W_F)-b_1^{\boldsymbol{\omega}}(X_L)+b_1^{\boldsymbol{\omega}}(W_F)-b_3^{\boldsymbol{\omega}}(W_F).$$
\end{enumerate}
\end{proposition}
\begin{proof}
We prove the first assertion.
The boundary of~$W_F$ decomposes as~$\partial W_F=X_L \cup_\partial M_F$, where~$M_F$ is a certain plumbed 3-manifold.
It can then be shown that~$H_*(M_F;\Lambda_S)=0$ (see~\cite[Lemma 5.2]{ConwayFriedlToffoli}) and the first assertion then follows from a Mayer-Vietoris sequence (use the third item of Lemma~\ref{lem:H0} to deduce that~$H_*(\partial X_L;\Lambda_S)=0$). 
The argument over~$\C^{\boldsymbol{\omega}}$ is analogous, as is mentioned in~\cite[Lemma~3.10]{ConwayNagelToffoli}.

We prove the second assertion.
Denote by~$\Im_{\boldsymbol{\omega}}$ the image of~$H_2(\partial W_F;\C^{\boldsymbol{\omega}}) \to H_2(W_F;\C^{\boldsymbol{\omega}})$ and note that~$\Im_{\boldsymbol{\omega}}$ lies in the radical of the intersection form~$\lambda_{W_F,\C^{\boldsymbol{\omega}}}$.
As~$F$ is a coloured bounding surface for~$L$, we have~$\sigma_L({\boldsymbol{\omega}})=\sigma_{\boldsymbol{\omega}}(W_F)$ by definition.
We deduce
\begin{align}
\label{eq:SignatureCoker}
|\sigma_L({\boldsymbol{\omega}})| 
&\leq \dim_\C \frac{H_2(W_F;\C^{\boldsymbol{\omega}})}{\dim(\Im_{\boldsymbol{\omega}})}. 
\end{align}
We compute the dimension of~$\Im_{\boldsymbol{\omega}}$.
Consider the following portion of the long exact sequence of the pair~$(W_F,\partial W_F)$ with~$\C^{\boldsymbol{\omega}}$ coefficients:
\begin{align*}
\label{eq:DimensionIm}
0 \to \Im_{\boldsymbol{\omega}} &\to H_2(W_F;\C^{\boldsymbol{\omega}}) \to H_2(W_F;\partial W_F;\C^{\boldsymbol{\omega}}) \to H_1(\partial W_F;\C^{\boldsymbol{\omega}})  \\
&\to H_1(W_F;\C^{\boldsymbol{\omega}}) \to H_1(W_F,\partial W_F;\C^{\boldsymbol{\omega}}) \to 0.
\end{align*}
By item i),~$H_1(\partial W_F;\C^{\boldsymbol{\omega}}) \cong H_1(X_L;\C^{\boldsymbol{\omega}})$. 
By Poincar\'{e} duality and the universal coefficient theorem, we have~$H_2(W_F,\partial W_F;\C^{\boldsymbol{\omega}}) \cong H_2(W_F;\C^{\boldsymbol{\omega}})$.
We also have~$H_1(W_F,\partial W_F;\C^{\boldsymbol{\omega}})\cong H_3(W_F;\C^{\boldsymbol{\omega}})$.
We deduce from these facts that the dimension of~$\Im_{\boldsymbol{\omega}}$ can be expressed as 
\begin{equation}
\label{eq:DimIm}
\dim_\C \Im_{\boldsymbol{\omega}}=b_1^{\boldsymbol{\omega}}(X_L)-b_1^{\boldsymbol{\omega}}(W_F)+b_3^{\boldsymbol{\omega}}(W_F).
\end{equation}
The second assertion now follows by combining~\eqref{eq:SignatureCoker} with~\eqref{eq:DimIm}.
\end{proof}

\subsection{A lower bound on the doubly slice genus}
\label{sub:ProofLowerBound}

Recall the \emph{doubly slice genus}~$g_{ds}(L)$ from Definition~\ref{def:DoublySlice}.
We now restate Theorem~\ref{thm:MainTheoremIntro} from the introduction, and prove it.

\begin{theorem}
\label{thm:LowerBound4D}
Given a~$\mu$-coloured link~$L$, for all~${\boldsymbol{\omega}} \in \mathbb{T}^\mu_*$ one has
$$|\sigma_L({\boldsymbol{\omega}})| \leq g_{ds}(L).$$
\end{theorem}
\begin{proof}
Assume that the~$\mu$-coloured link~$L \subset S^3$ is a cross section of a genus~$g$~$\mu$-component unlinked surface~$\Sigma \subset S^4$.
Decompose~$S^4$ as the union~$D^4_1 \cup_{S^3} D^4_2$ along its equatorial sphere, and define coloured bounding surfaces for~$L$ by setting~$F:=D^4_1 \cap \Sigma$ and~$G:=D^4_2 \cap \Sigma$.
It follows that the surface exterior~$X_\Sigma:=S^4 \setminus \nu \Sigma$ decomposes as~$X_\Sigma=W_F \cup_{X_L} W_G$, where~$W_F =D^4 \setminus \nu F$ and~$W_G=D^4 \setminus \nu G$ are the exteriors of the coloured bounding surfaces~$F$ and~$G$.
The $\C^\omega$ coefficient system on $X_\Sigma$ restricts to $W_F,W_G$ and $X_L$.
An Euler characteristic argument using Proposition~\ref{prop:HomologyUnlinkedSurface} shows that
\begin{equation}
\label{eq:EulerCharacForMainTheorem}
2(1-\mu)+2g=\chi(X_\Sigma)=-b_1^{\boldsymbol{\omega}}(W_F)-b_1^{\boldsymbol{\omega}}(W_G) +\sum_{i=F,G} \left( b_2^{\boldsymbol{\omega}}(W_i)-b_3^{\boldsymbol{\omega}}(W_i) \right).
\end{equation}
By Proposition~\ref{prop:HomologyUnlinkedSurface},~$H_1(X_\Sigma;\C^{\boldsymbol{\omega}})=\C^{\mu-1}$, and so the Mayer-Vietoris exact sequence for~$X_\Sigma$ yields
$$ H_2(X_\Sigma;\C^{\boldsymbol{\omega}}) \to H_1(X_L;\C^{\boldsymbol{\omega}}) \xrightarrow{\iota} H_1(W_F;\C^{\boldsymbol{\omega}}) \oplus H_1(W_G;\C^{\boldsymbol{\omega}}) \to \C^{\mu-1} \to 0.~$$
This implies that~$\dim_\C(\operatorname{im}(\iota)) \leq b_1^{\boldsymbol{\omega}}(X_L)$ and we deduce the inequality
\begin{equation}
\label{eq:InequalityForMainTheorem}
b_1^{\boldsymbol{\omega}}(W_F)+b_1^{\boldsymbol{\omega}}(W_G) \leq b_1^{\boldsymbol{\omega}}(X_L)+\mu-1.
\end{equation}
Applying the bound on~$\sigma_L({\boldsymbol{\omega}})$ from Proposition~\ref{prop:SignatureBound} twice (once for~$F$ and once for~$G$) and adding the results together, we obtain the following inequality:
\[
2|\sigma_L({\boldsymbol{\omega}})|  \leq -2b_1^{\boldsymbol{\omega}}(X_L)+( b_1^{\boldsymbol{\omega}}(W_F) +b_1^{\boldsymbol{\omega}}(W_G))+ \sum_{i=F,G} \left(b_2^{\boldsymbol{\omega}}(W_i)-b_3^{\boldsymbol{\omega}}(W_i) \right).
\]
Combining this inequality with~\eqref{eq:EulerCharacForMainTheorem} and (\ref{eq:InequalityForMainTheorem}) gives
\begin{align*}
2|\sigma_L({\boldsymbol{\omega}})| 
&\leq -2b_1^{\boldsymbol{\omega}}(X_L)+( b_1^{\boldsymbol{\omega}}(W_F) +b_1^{\boldsymbol{\omega}}(W_G))+ \sum_{i=F,G} \left(b_2^{\boldsymbol{\omega}}(W_i)-b_3^{\boldsymbol{\omega}}(W_i) \right) \\
& \leq -2b_1^{\boldsymbol{\omega}}(X_L)+2(b_1^{\boldsymbol{\omega}}(X_L)+\mu-1)+2(1-\mu)+2g\\
& =2g.
\end{align*}
Dividing both sides of this inequality by two yields the required result.
\end{proof}

\section{The Blanchfield pairing of doubly slice links}
\label{sec:Blanchfield}

In this section we will prove that the Blanchfield pairing of a strongly doubly slice link is hyperbolic, as stated in Theorem~\ref{thm:BlanchfieldStronglySliceIntro} from the introduction.

\subsection{The Blanchfield pairing of a coloured link}\label{sub:BlanchfieldColoured}

We first recall the definition of the Blanchfield pairing of a coloured link; references include~\cite{Hillman,BorodzikFriedlPowell,Conway}. Recall the rings~$\Lambda_S$ and~$Q$ defined in Subsection \ref{sec:introblanchfield}.
Let~$L=L_1 \cup \dots \cup L_\mu$ be a coloured link.
Consider the composition 
\begin{align*}
\widehat{\Bl}_L \colon  TH_1(X_L;\Lambda_S) \,
 & \xrightarrow{(i)} TH_1(X_L,\partial X_L;\Lambda_S)  \\
 & \xrightarrow{(ii)}\, \ker(H^2(X_L;\Lambda_S) \to H^2(X_L;Q))  \\
  & \xrightarrow{(iii)}   \frac{H^1(X_L;Q/\Lambda_S)}{\ker( H^1(X_L;Q/\Lambda_S) \stackrel{\text{BS}}{\to} H^2(X_L;\Lambda_S) )} \\
 & \xrightarrow{(iv)}  \overline{\Hom_{\Lambda_S}(TH_1(X_L;\Lambda_S),Q/\Lambda_S)}
\end{align*}
where the maps are as follows:
\begin{enumerate}[leftmargin=*]\setlength\itemsep{0em}
\item[$(i)$] The map~$H_1(X_L;\Lambda_S) \to H_1(X_L,\partial X_L;\Lambda_S)$ induced by the inclusion is an isomorphism (by the third item of Lemma~\ref{lem:H0}). Passing to the torsion submodules gives the first map, which is an isomorphism.
\item[$(ii)$] Since~$Q$ is flat over~$\Lambda_S$, we deduce that~$\ker(H^2(X_L;\Lambda_S) \to H^2(X_L;Q))=TH^2(X_L;\Lambda_S)$ and thus Poincar\'e duality induces the second map, which is an isomorphism.
\item[$(iii)$] The third map is the isomorphism obtained by considering the Bockstein exact seqeunce
\[
 \ldots \to H^1(X_L;Q) \to  H^1(X_L;Q/\Lambda_S) \xrightarrow{\operatorname{BS}} H^2(X_L;\Lambda_S)  \to H^2(X_L;Q) \to
\]
Indeed,~$TH^2(X_L;\Lambda_S)=\ker(H^2(X_L;\Lambda_S) \to H^2(X_L;Q))$ is equal to~$\im(BS) \cong \frac{H^1(X_L;Q/\Lambda_S)}{\ker(\text{BS})}$. 
\item[$(iv)$]  The fourth map is given by evaluation, and to justify it is well-defined we must show that elements of~$\ker(BS)$ evaluate to zero on elements of~$TH_1(X_L;\Lambda_S)$. Since~$\ker(BS)=\im(H^1(X_L;Q) \to H^1(X_L;Q/\Lambda_S))$, elements of~$\ker(BS)$ are represented by cocycles which factor through~$Q$-valued homomorphisms. Since~$Q$ is a field, these latter cocycles vanish on torsion elements, and thus so do the elements of~$\ker(BS)$.
\end{enumerate}

\begin{definition}
\label{def:BlanchfieldLinks}
The \textit{Blanchfield pairing} of a coloured link~$L$ is the pairing
$$ \Bl_L \colon  TH_1(X_L;\Lambda_S) \times  TH_1(X_L;\Lambda_S) \rightarrow Q/\Lambda_S ~$$
defined by~$\Bl_L (a,b)=\widehat{\Bl}_L(b)(a).$
\end{definition}

It follows from the definition of~$\widehat{\Bl}_L$ that the Blanchfield pairing is sesquilinear over~$\Lambda_S$, in the sense that~$\Bl_L(pa,qb)=p\Bl_L(a,b)\overline{q}$ for any~$a,b$ in~$TH_1(X_L;\Lambda_S)$ and any~$p,q$ in~$\Lambda_S.$
The Blanchfield pairing is also known to be Hermitian; see e.g.~\cite[Corollary 1.2]{Conway}.

\begin{remark}
\label{rem:BlanchfieldNonSingular}
The Blanchfield pairing of an ordered link~$L$ need not be nonsingular.
However, the Blanchfield pairing is known to be non-singular on~$\widehat{t}H_1(X_L;\Lambda_S)$.
Here, given a~$\Lambda_S$-module~$M$, we write 
$\widehat{t}M=TM/zM$ for the quotient of~$TM$ by its so-called \emph{maximal pseudo-null submodule}~$zM$~\cite[p.~30]{Hillman}.

In fact, when~$L$ is a boundary link, one has~$\widehat{t}H_1(X_L;\Lambda_S)=TH_1(X_L;\Lambda_S)$~\cite[Lemma~2.2]{Hillman}.
In particular, the Blanchfield pairing of a boundary link is non-singular. We prefer to give a more direct proof of this fact as follows.
\end{remark}

\begin{proposition}
\label{prop:Nonsingular}

If the $\Lambda_S$-Alexander module of a $\mu$-coloured link $L$ admits a direct sum decomposition of the form $H_1(X_L;\Lambda_S)=TH_1(X_L;\Lambda_S) \oplus \Lambda_S^b$ for some $b$, then the Blanchfield pairing of~$L$ is non-singular.

\end{proposition}
\begin{proof}
We have already argued that the maps~$(i),(ii)$ and~$(iii)$ in the definition of the Blanchfield pairing are isomorphisms.
We therefore only need to prove that the evaluation map~$(iv)$ is an isomorphism.
Given a~$\Lambda_S$-module~$H$, set~$H^*:=\overline{\Hom_{\Lambda_S}(H,\Lambda_S)},H^{\#}:=\overline{\Hom_{\Lambda_S}(H,Q)}$ as well as~$H^\vee:=\overline{\Hom_{\Lambda_S}(H,Q/\Lambda_S)}$.
Set~$X:=X_L$, consider the following commutative diagram:
\[
\xymatrix@R0.5cm@C0.5cm{
0\ar[r]&H^1(X;\Lambda_S)\ar[r]\ar[d]^{\ev}_\cong&
H^1(X;Q)\ar[r]\ar[d]^{\ev}_\cong&
H^1(X;Q/\Lambda_S)\ar[r]\ar[d]^{\ev}_\cong&
H^1(X;Q/\Lambda_S)/\ker(BS) \ar[r]\ar[d]^{\ev}&
 0 \\
0\ar[r]& H_1(X;\Lambda_S)^* \ar[r]&
H_1(X;\Lambda_S)^{\#} \ar[r]^-{\iota_Q^{Q/\Lambda_S}}&H_1(X;\Lambda_S)^\vee \ar[r]^{\operatorname{incl}^\vee}&TH^1(X;\Lambda_S)^\vee \ar[r] & 0. \\
}
\]
Since~$H_0(X_L;\Lambda_S)=0$, the universal coefficient spectral sequence gives the three vertical isomorphisms. The top row is exact, arising from the Bockstein long exact sequence. We will prove that the bottom row is also exact, for then 
the five lemma
shows the right-most vertical arrow is an isomorphism. 
This vertical arrow is the map labelled~$(iv)$ in the definition of~$\widehat{\Bl}_L$ so this will complete the proof of the proposition.

Set~$H:=H_1(X_L;\Lambda_S)$. 
By assumption, we have~$H/TH\cong~\Lambda_S^{b}$, so~$\operatorname{Ext}_{\Lambda_S}^1(H/TH,Q/ \Lambda_S)=0$, and hence~$\operatorname{incl}^\vee$ is surjective.
Since all the evaluation maps but one are isomorphisms, it only remains to prove exactness at~$H^1(X;\Lambda_S)^\vee$, i.e.~to establish
\begin{equation}
\label{eq:WantExact}
 \im(i_Q^{Q/\Lambda_S})=\ker(\operatorname{incl}^\vee).
 \end{equation}
By our assumption on the coloured link~$L$, we have~$H=TH \oplus \Lambda_S^{b}$ and so~$H^*=\Lambda_S^{b}$ and~$H^{\#}=Q^{ b }$ so that~$\im(\iota_Q^{Q/\Lambda_S})=\overline{\Hom_{\Lambda_S}(\Lambda_S^{ b },Q/\Lambda_S)}$.
We now describe~$\ker(\operatorname{incl}^\vee)$. 
Apply~$\overline{\Hom_{\Lambda_S}(-,Q/\Lambda_S)}$ to the short exact sequence~$0 \to TH \xrightarrow{\operatorname{incl}} H \to H/TH \to 0$.
As~$H/TH=\Lambda_S^{ b }$, the result is also~$\ker(\operatorname{incl}^\vee)=\overline{\Hom_{\Lambda_S}(\Lambda_S^{ b },Q/\Lambda_S)}$, verifying~\eqref{eq:WantExact} as required.
\end{proof}

Combining Corollary~\ref{cor:SplitModuleColored} and Proposition~\ref{prop:Nonsingular}, we obtain the following.

\begin{corollary}
\label{cor:Nonsingular}
If $L$ is an $n$-component boundary link that is endowed with a $\mu$-colouring, then the Blanchfield pairing $\Bl_L$ is non-singular.
\end{corollary}

We conclude with a remark specific to the one-variable setting.
Given an oriented link, the exact same construction as in Definition~\ref{def:BoundaryLink} produces a \emph{rational Blanchfield pairing}
$$\Bl_L^\Q \colon TH_1(X_L;\Q[t^{\pm 1},(1-t)^{-1}]) \times TH_1(X_L;\Q[t^{\pm 1},(1-t)^{-1}]) \to \Q(t)/\Q[t^{\pm 1}].$$
Since the ring $\Q[t^{\pm 1},(1-t)^{-1}]$ is obtained by localising the PID $\Q[t^{\pm 1}]$, it is itself a PID.
It follows that $H_1(X_L;\Q[t^{\pm 1},(1-t)^{-1}])$ splits into a free and torsion summand, and therefore the exact same proof as the one of Proposition~\ref{prop:Nonsingular} yields the following result.
\begin{corollary}
\label{cor:Nonsingular}
If $L$ is an oriented link, then the rational Blanchfield pairing $\Bl_L^\Q$ is non-singular.
\end{corollary}

We conclude with one final remark about boundary links in the case $\mu=1$.
\begin{remark}
\label{rem:FellerParkPowell}
If $L$ is an oriented boundary link and $\mu=1$, then multiplication by $(t-1)$ induces an isomorphism on $TH_1(X_L;\Z[t^{\pm 1}])$~\cite[Lemma 3.3]{FellerParkPowell}.
In particular, in this case, the chain map $C_*(\widehat{X}_L) \to \Z[t^{\pm 1},(1-t)^{-1}] \otimes_{\Z[t^{\pm 1}]} C_*(\widehat{X}_L)$ induces an isomorphism
\[
TH_1(X_L;\Z[t^{\pm 1}]) \cong TH_1(X_L;\Z[t^{\pm 1},(1-t)^{-1}])
\]
that preserves the Blanchfield forms.
\end{remark}

\subsection{The Blanchfield pairing of a strongly doubly slice link is hyperbolic}

We are now able to prove Theorem~\ref{thm:BlanchfieldStronglySliceIntro} and Proposition~\ref{prop:ModuleSplit} from the introduction. 
\begin{theorem}
\label{thm:BlanchfieldStronglySlice}

Let $L$ be an $n$-component boundary link that is endowed with a $\mu$-colouring.
If $L$ is doubly slice as a $\mu$-coloured link, then
\begin{enumerate}[leftmargin=*]\setlength\itemsep{0em}
\item the Blanchfield pairing of $L$ is hyperbolic;
\item for some $\Lambda_S$-submodules $G_1,G_2 \subset TH_1(X_L;\Lambda_S)$ with $\operatorname{Ord}(G_1)=\overline{\operatorname{Ord}(G_2)}$, the torsion submodule of the~$\Lambda_S$-Alexander module of $L$ splits as $TH_1(X_L;\Lambda_S)=G_1 \oplus G_2.$
\end{enumerate}
By Remark~\ref{rem:FellerParkPowell}, for $\mu=1$, these results hold not only over $\Z[t^{\pm 1},(1-t)^{-1}]$, but also over $\Z[t^{\pm 1}]$.

Furthermore, for any weakly doubly slice oriented link $L$, the rational Blanchfield form $\Bl_L^\Q$ is hyperbolic and $TH_1(X_L;\Q[t^{\pm 1},(1-t)^{-1}])=G\oplus {\overline{G}}$ for some $\Q[t^{\pm 1},(1-t)^{-1}]$-module $G$.
\end{theorem}

Note in particular, the hypotheses of Theorem \ref{thm:BlanchfieldStronglySlice} are satisfied when $L$ is a strongly doubly slice link and when $L$ is a weakly doubly slice boundary link.

\begin{proof}
Assume that the coloured link~$L \subset S^3$ is a cross section of an unlink~$\Sigma \subset S^4$.
Decompose~$S^4$ as the union~$D^4_1 \cup_{S^3} D^4_2$ along its equatorial sphere, and define coloured bounding surfaces for~$L$ by setting~$F:=D^4_1 \cap \Sigma$ and~$G:=D^4_2 \cap \Sigma$.
It follows that the unlink exterior~$X_\Sigma:=S^4 \setminus \nu \Sigma$ decomposes as~$X_\Sigma=W_F \cup_{X_L} W_G$, where~$W_F =D^4 \setminus \nu F$ and~$W_G=D^4 \setminus \nu G$ are the exteriors of the coloured bounding surfaces~$F$ and~$G$.
We therefore obtain the following Mayer-Vietoris exact sequence with~$\Lambda_S$ coefficients:
\begin{equation}
\label{eq:MVLambdaS}
0 \to H_1(X_L;\Lambda_S) \xrightarrow{\bsm \iota_F \\ \iota_G \esm} H_1(W_F;\Lambda_S) \oplus H_1(W_G;\Lambda_S) \to H_1(X_\Sigma;\Lambda_S) \to 0.
\end{equation}
Since $L$ is a boundary link, we know from Corollary~\ref{cor:SplitModuleColored} that~$H_1(X_L;\Lambda_S)=TH_1(X_L;\Lambda_S) \oplus \Lambda_S^{n-1}$.
The same conclusion holds over the PID $\Q[t^{\pm 1},(1-t)^{-1}]$ for any oriented link $L$.
We also know from Lemma~\ref{lem:WedgeOfCirclesLambdaS} that~$H_1(X_\Sigma;\Lambda_S)=\Lambda_S^{\mu-1}$.
Consequently the short exact sequence in~\eqref{eq:MVLambdaS} splits and, in particular we deduce that  
\begin{align}
\label{eq:SplittingFG}
TH_1(X_L;\Lambda_S)&\cong TH_1(W_F;\Lambda_S) \oplus TH_1(W_G;\Lambda_S),
\end{align}
Here note that since~$H_1(X_L;\Lambda_S)$ splits into free and torsion parts, so does~$ H_1(W_F;\Lambda_S) \oplus  H_1(W_G;\Lambda_S)$ and therefore so do~$H_1(W_F;\Lambda_S)$ and~$ H_1(W_G;\Lambda_S)$.
From now on, we will write~$T\iota_F$ and~$T\iota_G$ for the restriction of~$\iota_F$ and~$\iota_G$ to the torsion submodule of the~$\Lambda_S$-Alexander module of~$L$:
\begin{align*}
T\iota_F \colon TH_1(X_L;\Lambda_S) \to TH_1(W_F;\Lambda_S), \\
T\iota_G \colon TH_1(X_L;\Lambda_S) \to TH_1(W_G;\Lambda_S).
\end{align*}
We now have~$TH_1(X_L;\Lambda_S)=\ker(T\iota_F) \oplus \ker(T\iota_G)$, and our objective is to show that these two kernels are metabolisers for the Blanchfield pairing~$\Bl_L$: this will show that~$\Bl_L$ is hyperbolic.

We prove that~$\ker(T \iota_F)$ is a metaboliser of~$\Bl_L$; the proof for~$\ker(T\iota_G)$ is identical.
We follow the strategy from~\cite[Theorem 4.4 and Lemma 4.5]{CochranOrrTeichner}, applying it in a similar way to \cite[Proposition 2.10]{KimNew}.

\begin{claim}
\label{claim:Inj}
The following sequence is exact:
\[
TH_2(W_F,\partial W_F;\Lambda_S) \xrightarrow{\partial} TH_1(\partial W_F;\Lambda_S) \xrightarrow{T\iota_F}  TH_1(W_F;\Lambda_S) \to 0.
\]
\end{claim}
\begin{proof}[Proof of claim]
Observe that~\eqref{eq:SplittingFG} shows~$T\iota_F$ is surjective. We must show exactness at~$TH_1(\partial W_F;\Lambda_S)$.
We first assert that~$H_2(W_F,\partial W_F;Q) \to H_1(\partial W_F;Q)$ is injective. 
Since~$H_2(X_\Sigma;Q)=0$ (as~$X_\Sigma$ is a connected sum of $S^1 \# D^3$.)
using the Mayer-Vietoris sequence for~$X_\Sigma=W_F\cup_{X_L} W_G$, we know that the inclusion induced map~$H_2(X_L;Q) \twoheadrightarrow H_2(W_F;Q) \oplus H_2(W_G;Q)$ is surjective. 
Since we have~$H_1(\partial W_F;Q)=H_i(X_L;Q)$ by Proposition~\ref{prop:SignatureBound}, we deduce that~$H_2(\partial W_F;Q) \twoheadrightarrow H_2(W_F;Q)$ is surjective.
By exactness, this implies that~$H_2(W_F;Q) \to H_2(W_F,\partial W_F;Q)$ is the zero map.
By exactness, this implies that~$H_2(W_F;\partial W_F;Q) \to H_1(\partial W_F;Q)$ is injective, as asserted

As in~\cite[proof of Theorem 1.1]{BorodzikFriedlPowell} we consider the following commutative diagram with exact columns:
$$
\xymatrix@R0.5cm{
&0\ar[d]&0\ar[d]&  0\ar[d] \\
&TH_2(W_F,\partial W_F;\Lambda_S)\ar[d]\ar[r]&TH_1(\partial W_F;\Lambda_S) \ar[d]\ar[r]&  TH_1(W_F;\Lambda_S)\ar[d] \\
&H_2(W_F,\partial W_F;\Lambda_S)\ar[d]\ar[r]&H_1(\partial W_F;\Lambda_S)\ar[d]\ar[r] &  H_1(W_F;\Lambda_S)\ar[d] \\
0\ar@{->}[r]&H_2(W_F,\partial W_F;Q)\ar[r]&H_1(\partial W_F;Q) \ar[r]&  H_1(W_F;Q). \\
}
$$
The middle row is exact and the bottom row is exact thanks to the assertion.
By the ``sharp~$3 \times 3$ lemma"~\cite[Lemma 2]{FHH}, it follows that the top row is exact, establishing the claim.
\end{proof}

We wish to define a pairing~$\Bl_F\colon TH_1(W_F;\Lambda_S)\times TH_2(W_F,\partial W_F;\Lambda_S)\to Q/\Lambda_S$. For this, consider the following commutative diagram. The pairing we desire is the adjoint of the leftmost column. 
\[
\xymatrix@R0.7cm{
TH_2(W_F,\partial W_F;\Lambda_S) \ar[r]^-{\partial}\ar[d]_-\cong^-{\PD}& TH_1(\partial W_F;\Lambda_S)\ar[r]^-{T\iota_F}\ar[d]_-\cong^-{\PD}&TH_1(W_F;\Lambda_S)\ar[r]&0  \\
TH^2(W_F;\Lambda_S) \ar[d]^-{\operatorname{BS}^{-1}}_-\cong \ar[r]^-{\iota_F^*}& TH^2(\partial W_F;\Lambda_S) \ar[d]^-{\operatorname{BS}^{-1}}_-\cong \\
\frac{TH^1(W_F;Q/\Lambda_S)}{\ker\operatorname{BS}} \ar[d]^-{\ev}_-\cong \ar[r]^-{\iota_F^*}& \frac{TH^1(\partial W_F;Q/\Lambda_S)}{\ker\operatorname{BS}}  \ar[d]^-{\ev}_-\cong  \\
\overline{\Hom_{\Lambda_S}(TH_1(W_F;\Lambda_S),Q/\Lambda_S)}  \ar[r]^-{T\iota_F^\vee} & \overline{\Hom_{\Lambda_S}(TH_1(\partial W_F;\Lambda_S),Q/\Lambda_S)}. 
}
\]
The top row is the exact sequence from Claim~\ref{claim:Inj}.
We briefly justify this various isomorphisms in the diagram. 
The vertical maps~$\operatorname{BS}^{-1}$ are defined as the map~$(iii)$ in the definition of~$\widehat{\Bl}_L$ and are justified to be isomorphisms similarly.
The bottom vertical map in the central column is justified to be an isomorphism similarly to the proof of Proposition~\ref{prop:Nonsingular}. Similarly, the proof that the bottom left evaluation map is an isomorphism uses the splitting~$H_1(W_F;\Lambda_S)=TH_1(W_F;\Lambda_S) \oplus \Lambda_S^b$ for some~$b$; just as in the proof of Proposition~\ref{prop:Nonsingular}. We omit the repetition of the details.

Now set~$P:=\ker(T\iota_F)=\im(\partial)$. We claim that~$P=P^\perp$.
The inclusion~$P \subset P^\perp$ follows from the commutativity of the diagram above: given~$\partial(z) \in \im(\partial)=P$ and~$x \in \ker(T\iota_F)=P$, one has~$\Bl_L(x,\partial(z))=\operatorname{Bl}_F(T\iota_F(x),z)=0$.
It remains to prove the reverse inclusion, namely~$P^\perp \subset P$.
Let~$x \in P^\perp$ so that~$\Bl_L(x)(p)=0$ for all~$p \in P$.
In other words,~$\Bl(x)$ defines an element in~$\overline{\Hom(TH_1(\partial W_F;\Lambda_S)/P,Q/\Lambda_S )}$.
By Claim~\ref{claim:Inj}, we know that~$T\iota_F$ is surjective and therefore, since~$P=\ker(T\iota_F)$, it induces an isomorphism~$TH_1(\partial W_F;\Lambda_S)/P \cong TH_1(W_F;\Lambda_S)$.
It follows that we obtain an isomorphism
$$T\iota_F^\vee \colon \overline{\Hom(TH_1(W_F;\Lambda_S),Q/\Lambda_S )} \xrightarrow{\cong} \overline{\Hom(TH_1(\partial W_F;\Lambda_S)/P,Q/\Lambda_S )}.$$
Thus~$\Bl(x)$ lives in the image of~$T\iota^\vee_F$.
By the commutativity of diagram above (and especially the fact that all of the left vertical maps are isomorphisms), it follows that~$x \in \im(\partial)=P$.

Summarising, we have proved that~$\ker(T\iota_F)=\ker(T\iota_F)^\perp$.
The proof that~$\ker(T\iota_G)=\ker(T\iota_G)^\perp$ is identical and we have therefore established that~$\ker(T\iota_F)$ and~$\ker(T\iota_G)$ are metabolisers of~$\Bl_L$ with~$\ker(T\iota_F) \oplus \ker(T\iota_G)=TH_1(X_L;\Lambda_S)$.
This concludes the proof that~$\Bl_L$ is hyperbolic.

It remains to prove that $\operatorname{Ord}(\ker(T \iota_F))=\overline{\operatorname{Ord}(\ker(T \iota_G)})$.
First, note that the isomorphism displayed in~\eqref{eq:SplittingFG} implies that $\ker(T\iota_G)=TH_1(W_F;\Lambda_S)$.
The left vertical column in the previous commutative diagram then shows that 
\begin{align*}
\ker(T \iota_F)
&=TH_2(W_F,\partial W_F;\Lambda_S)
\cong \overline{\Hom_{\Lambda_S}(TH_1(W_F;\Lambda_S),Q/\Lambda_S)} \\
&\cong \overline{\Hom_{\Lambda_S}(\ker(T\iota_G),Q/\Lambda_S)}.
\end{align*}
As $\ker(T\iota_G)$ is torsion, we deduce that $\overline{\Hom_{\Lambda_S}(\ker(T\iota_G),Q/\Lambda_S)}$ is isomorphic to $\overline{\operatorname{Ext}^1_{\Lambda_S}(\ker(T\iota_G),\Lambda_S)}$.
By~\cite[Lemma 15.7 items (3) and (4)]{FriedlNagelOrsonPowell}, one deduces that the order of this latter module equals $\overline{\operatorname{Ord}(\ker(T\iota_G))}$; here note that~\cite[Lemma 15.7]{FriedlNagelOrsonPowell} is stated for $\Lambda$, but the proofs go through for any localisation of $\Lambda$.

The exact same proof works for $\Q[t^{\pm 1},(1-t)^{-1}]$ in place of $\Lambda_S$, but now since $\Q[t^{\pm 1},(1-t)^{-1}]$ is a PID, $\operatorname{Ord}(\ker(T \iota_F))=\overline{\operatorname{Ord}(\ker(T \iota_G)})$ implies that~$\ker(T \iota_F)=\overline{\ker(T \iota_G)}$.
This concludes the proof of the theorem. 
\end{proof}

The same methods as those used in the proof of Theorem~\ref{thm:BlanchfieldStronglySlice} can also be applied to the linking form on the double branched cover of $L$.
The linking form of a compact, connected, oriented $3$-manifold $M$ is a symmetric bilinear pairing $\lambda_M \colon TH_1(M) \times TH_1(M) \to \Q/\Z$.
When $M$ is closed, $\lambda_M$ is nonsingular, for instance by using the same arguments as in Proposition~\ref{prop:Nonsingular} with $\Z$ and~$\Q$ in place of $\Lambda_S$ and~$Q$.
In particular, the $2$-fold branched cover along a link $L$ comes with a non-singular linking form $\lambda_{\Sigma_2(L)}$.
The next result proves Proposition~\ref{prop:LinkingFormIntro} from the introduction.

\begin{proposition}\label{prop:doublebranched}
If $L$ is weakly doubly slice, then $\lambda_{\Sigma_2(L)}$ hyperbolic.
Furthermore, for some finite group $G$ we have 
$TH_1(\Sigma_2(L))=G \oplus G.$
\end{proposition}
\begin{proof}
Assume that the ordered link~$L \subset S^3$ is a cross section of a $2$-sphere~$S \subset S^4$.
Decompose~$S^4$ as the union~$D^4_1 \cup_{S^3} D^4_2$ along its equatorial sphere, and define coloured bounding surfaces for~$L$ by setting~$F:=D^4_1 \cap S$ and~$G:=D^4_2 \cap S$.
It follows that $ \Sigma_2(S)=\Sigma_2(F) \cup_{\Sigma_2(L)} \Sigma_2(G)$, where $\Sigma_2(F)$ and $\Sigma_2(G)$ respectively denote the $2$-fold covers of $D^4$ branched along $F$ and $G$.
Since $S$ consists of a single $2$-sphere, we deduce that the 2-fold branched cover $\Sigma_2(S)$ is diffeomorphic to $S^4.$
A Mayer-Vietoris argument therefore shows that the inclusion induced maps induce the following isomorphism:
\begin{equation}
\label{eq:MVLambdaS}
\begin{pmatrix} \iota_F \\ \iota_G \end{pmatrix}  \colon H_1(\Sigma_2(L)) \xrightarrow{\cong} H_1(\Sigma_2(F)) \oplus H_1(\Sigma_2(G)).
\end{equation}
Recall that every abelian group splits as a free part and a torsion part.
Consequently the previous isomorphism from~\eqref{eq:MVLambdaS} restricts to an isomorphism on torsion summands:
\begin{align*}
\begin{pmatrix} T\iota_F \\ T\iota_G \end{pmatrix}  \colon TH_1(\Sigma_2(L)) \xrightarrow{\cong} TH_1(\Sigma_2(F)) \oplus TH_1(\Sigma_2(G)).
\end{align*}
Here, we wrote~$T\iota_F$ and~$T\iota_G$ for the restriction of~$\iota_F$ and~$\iota_G$ to the torsion subgroups.
The same argument as in the proof of Theorem~\ref{thm:BlanchfieldStronglySlice}, now shows that $TH_1(\Sigma_2(F))=\ker(T\iota_G)$ and $TH_1(\Sigma_2(G))=\ker(T\iota_F)$ are metabolisers for the linking form~$\lambda_{\Sigma_2(L)}$, and this shows that~$\lambda_{\Sigma_2(L)}$ is hyperbolic.

It remains to show that $TH_1(\Sigma_2(F))$ and $TH_1(\Sigma_2(G))$ are isomorphic abelian groups.
This will follow from the following sequence of isomorphisms
\begin{align*}
TH_1(\Sigma_2(F)) 
&\xleftarrow{\partial,\cong} TH_2(\Sigma_2(S),\Sigma_2(F)) \\
&\cong TH_2(\Sigma_2(G),\Sigma_2(L)) \\
&\cong TH^2(\Sigma_2(G)) \\
&\cong \operatorname{Ext}^1_\Z(TH_1(\Sigma_2(G)),\Z) \\
&\cong TH_1(\Sigma_2(G)).
\end{align*}
The first isomorphism comes from the long exact sequence of the pair (using that $\Sigma_2(S) \cong S^4$), the second is excision, the third is Poincar\'{e}-Lefschetz duality, the fourth is the universal coefficient theorem and the last follows from properties of the Ext functor.
\end{proof}

\section{The Seifert surface approach}
\label{sec:Seifert}

The goal of this section is to prove the general coloured version of Theorem~\ref{thm:SeifertMatrixIntro} from the introduction about Seifert matrices of doubly slice links.
In Subsections~\ref{sub:BoundarySeifertMatricesAlg} and~\ref{sub:BoundarySeifertSurfaces} we describe a coloured version of the boundary Seifert matrices of~\cite{KoSeifert} and how such a collection of matrices can be associated to a \emph{boundary} coloured link.
In Subsection~\ref{sub:DoublyMetabolic}, we prove that if a coloured link is doubly slice then it admits a \emph{doubly isotropic} coloured boundary Seifert matrix.
In Subsection~\ref{sub:Ccomplex}, we show how this result can be used to recover results about multivariable invariants.

\subsection{Boundary Seifert matrices}
\label{sub:BoundarySeifertMatricesAlg}
A matrix~$A$ over~$\Z$ is a called a \emph{Seifert matrix for a knot} if~$A-A^T$ is invertible. 
A matrix~$A$ over~$\Z$ is a \emph{Seifert matrix for a link} if~$A-A^T$ is congruent to the block sum of a nonsingular matrix and a~$0$ matrix.

\begin{definition}
\label{def:BoundarySeifertMatrix}
A matrix~$A=(A_{ij})_{1 \leq i , j \leq \mu}$ with entries given by a collection of~$\mu^2$ matrices~$A_{ij}$ over~$\Z$ is called a \emph{coloured boundary Seifert matrix} if for some~$a_i\in\mathbb{N}$ we have
\begin{enumerate}[leftmargin=*]\setlength\itemsep{0em}
\item for~$i=1,\ldots,\mu$, the matrix~$A_{ii}$ is a 
{size}~$a_i$ Seifert matrix for a link;
\item for~$i \neq j$, the matrix~$A_{ij}$ is an~$(a_i \times a_j)$-matrix with~$A_{ij}=A_{ji}^T$.
\end{enumerate}
If the~$A_{ii}$ are Seifert matrices for knots, then we call~$A$ a \emph{boundary Seifert matrix.}
A coloured boundary Seifert matrix~$A=(A_{ij})$ has an associated collection of~$\mu^2$ bilinear pairings between based free~$\Z$-modules~$H_i$ of rank~$a_i$, namely as~$ A_{ij} \colon H_i \times H_j \to \Z.$
\end{definition}

In the case where the~$A_{ii}$ are Seifert matrices for knots, Definition~\ref{def:BoundarySeifertMatrix} recovers the notion of a boundary Seifert matrix due to Ko~\cite[p. 668]{KoSeifert}.

\begin{definition}
\label{def:Metaboliser}
Let~$A=(A_{ij})_{1 \leq i , j \leq \mu}$ be a coloured boundary Seifert matrix with associated pairings~$A_{ij} \colon H_i \times H_j \to~\Z$.
\begin{enumerate}[leftmargin=*]\setlength\itemsep{0em}
\item An \emph{isotropic family} for~$A$ is a collection~$G=\{G_i\}_{1 \leq i \leq \mu}$ direct summands~$G_i\subset H_i$,
such that for all~$i,j$ the restriction of~$A_{ij}$ to~$G_i\times G_j$ is the trivial pairing.
\item If~$A$ admits two isotropic families~$G^\pm$ such that~$G_i^- \oplus G_i^+=H_i$ for each~$i$, then we call~$A$ \emph{doubly isotropic}.
\end{enumerate}
\end{definition}

\begin{remark}
\label{rem:MetabolicIsMetabolic}
We make some comments justifying the terminology.
\begin{enumerate}[leftmargin=*]\setlength\itemsep{0em}
\item An \emph{isotropic submodule} for a pairing~$A \colon \Z^{2g} \times \Z^{2g} \to \Z$ is a direct summand~$G \subset H$ such that the restriction of~$A$ to~$G$ is the trivial pairing;~$G$ is a \emph{metaboliser} if, additionally,~$G$ is half rank. 
If~$G=\lbrace G_i \rbrace_{1 \leq i \leq \mu}$ is an isotropic family for a coloured boundary Seifert matrix~$A=(A_{ij})$, then~$G_1 \oplus \ldots \oplus G_\mu$ is an isotropic submodule for~$A$.
\item A pairing~$A \colon \Z^{2g} \times \Z^{2g} \to \Z$ is called \emph{hyperbolic} if it admits two metabolisers~$G^{\pm}$ such that~$H=G^-\oplus G^+$; if the~$G^{\pm}$ are merely isotropic, then we call~$A$ \emph{doubly isotropic}.
Note that the diagonal blocks of a doubly isotropic matrix might in general not be of the same size; it is for this reason we chose to introduce the new terminology `doubly isotropic'.
If~$G^{\pm}=\lbrace G_i^{\pm} \rbrace_{1 \leq i \leq \mu}$ are isotropic families for a coloured boundary Seifert matrix~$A=(A_{ij})$, as in the second item of Definition~\ref{def:Metaboliser}, then~$A$ is doubly isotropic.
\item We observe that if a boundary Seifert matrix is doubly isotropic then for each~$i$, the matrix~$A_{ii}$ is in fact hyperbolic. To see this, consider that if a pairing~$A$ on~$H$ is doubly isotropic with respect to a decomposition~$G^-\oplus G^+\cong H$, then so is~$A-A^T$. But by definition~$A-A^T$ is nonsingular when~$A$ comes from a Seifert matrix for a knot. An isotropic summand of a nonsingular pairing can be at most half rank. As the ranks of~$G^+$ and~$G^-$ add up to the rank of~$H$, they each have exactly half rank and so~$A$ is hyperbolic.
\item If a boundary Seifert matrix $A$ is doubly isotropic, via $G_1^{\pm},\ldots, G_n^{\pm}$ then, as an integral matrix, it is hyperbolic, with orthogonal metabolisers $G^{\pm}:=G_1^\pm \oplus \ldots \oplus G_n^\pm$.
To see, this combine the first and third items above.
\item
According to our convention, every coloured boundary Seifert matrix admits an isotropic family, namely the \emph{trivial collection} $G=\{G_i\}$ where $G_i=\{0\}\subset H_i$ for all $i$. Complementary isotropic families arise naturally in the study of doubly slice links (see Theorem~\ref{thm:HyperbolicSeifert}), and in practice it is certainly possible that one of these naturally occurring families is the trivial collection.
\end{enumerate}
\end{remark}

\subsection{Boundary Seifert surfaces}
\label{sub:BoundarySeifertSurfaces}

In this subsection, we introduce the natural coloured generalisation of the notion of a boundary link.
We then show how to obtain a coloured boundary Seifert matrix for this generalisation.

\begin{definition}
\label{def:BoundaryLink}
A~$\mu$-coloured link~$L=L_1 \cup \ldots \cup L_\mu$ is \emph{boundary} if there are disjoint Seifert surfaces~$F_1,\ldots,F_\mu$
containing no closed components (but possibly disconnected),
such that $\partial F_i=L_i$. 
The collection~$F=F_1 \sqcup \ldots \sqcup F_\mu$ of Seifert surfaces is called a \emph{boundary Seifert surface for~$L$}.
\end{definition}

For~$\mu=n$, Definition~\ref{def:BoundaryLink} recovers the usual notion of a boundary link, while for~$\mu=1$ in fact any oriented link satisfies the definition.
In our context, the key observation is that when a coloured link is doubly slice, then it is boundary in the sense above.

\begin{remark}As the reader familiar with boundary links might expect, Definition \ref{def:BoundaryLink} has a more algebraic characterisation as follows. Let $X_L$ be a $\mu$-coloured link together with a homomorphism $\phi\colon\pi_1(X_L)\to F\langle s_1,\dots,s_\mu\rangle$, the free group on $\mu$ letters, that for each $i$ sends every $i$-coloured meridian to $s_i$. This homomorphism determines a map $X_L\to\vee_{i=1}^\mu(S^1)_i$, up to homotopy. For each $i$ choose a regular point $x_i\in(S^1)_i$, away from the wedge point, and take a transverse preimage. This will determine a boundary Seifert surface $F_\phi$ for $L$.

Conversely, given a boundary Seifert surface $F$ write $m_F\colon X_L\to \vee_i (S^1)_i$ for the continuous map given by
\[
\operatorname{int}(F_i)\times[-1,1]\xrightarrow{\operatorname{pr}_2}[-1,1]\xrightarrow{e} (S^1)_i,
\]
where $e(x):=\exp(\pi i(x+1))$, and sending $X_L\sm  (\operatorname{int}(F_i)\times[-1,1])$ to the basepoint. The induced map on fundamental groups determines a homomorphism~$\phi$ of the type described above, and such that $F=F_\phi$ when we choose the points $x_i$ to be $m_F(F_i\times\{0\})$.
\end{remark}

Next, we describe how to associate a coloured boundary Seifert matrix (as in Definition~\ref{def:BoundarySeifertMatrix}) to a boundary~$\mu$-coloured link.
Let~$F=F_1\sqcup \ldots \sqcup F_\mu$ be a boundary Seifert surface for a boundary~$\mu$-coloured link~$L=L_1 \cup \ldots \cup L_\mu$.
Pushing curves off~$F$ in the negative normal direction then produces a homomorphism~$i_- \colon H_1(F) \to H_1(S^3 \setminus F)$. 
The assignment
$$\theta(x,y):=\lk(i_-(x),y)$$ gives rise to a pairing on~$H_1(F)$ and to a \textit{coloured boundary Seifert matrix for~$L$} (the details are identical to \cite[p.670]{KoSeifert} where the~$\mu=n$ case is treated).
Since~$H_1(F)$ decomposes as the direct sum of the~$H_1(F_i)$, by choosing bases for the surfaces~$F_i$, the restriction of~$\theta$ to~$H_1(F_i) \times H_1(F_j)$ produces matrices~$A_{ij}$. 
For~$i \neq j$, these matrices satisfy~$A_{ij}=A_{ji}^T$, while ~$A_{ii}$ is a Seifert matrix for the sublink~$L_i$. 
It follows that~$A$ is a coloured boundary Seifert matrix for~$L$ in the sense of Definition~\ref{def:BoundarySeifertMatrix}.

\subsection{Doubly slice links have doubly isotropic Seifert matrices}
\label{sub:DoublyMetabolic}

In this subsection, we prove the general coloured version of Theorem~\ref{thm:SeifertMatrixIntro} from the introduction: we show that doubly slice coloured links have doubly isotropic coloured boundary Seifert matrices.

We begin with a technical lemma that will help us avoid a thorny issue related to the fact that our definition of boundary Seifert surfaces does not permit closed components.

\begin{lemma}\label{lem:partitionargument}
Suppose $M\subset S^4$ is an embedded $3$-ball meeting an equatorial $S^3$ nontrivially and transversely. Write $M\cap S^3=F\cup S$, where $S$ is closed and $F$ has no closed components. Then there exist submanifolds $M_\pm\subset M$ such that $M_+\cap M_-=F$ and $M=M_+\cup_F M_-$.
\end{lemma}

\begin{proof} Use the equator $S^3$ to write a decomposition $S^4=D_+^4\cup_{S^3}D^4_-$ into hemispheres. Write $L=\partial F\subset \partial M\cong S^2$. Observe that $S^2\cong \partial M=X_+\cup_L X_-$ where $X_\pm:=D_\pm^4\cap \partial M$.
Note that as $\partial M\cong S^2$ meets $S^3$ transversely, no connected component of $X_+\subset S^2$ shares a boundary component with any another component of $X_+$.
For each connected component of $M\sm F$ and of $\partial M\sm L$ choose a basepoint. Let $R$ be a connected component of $M\sm F$ and $T$ be a connected component of $\operatorname{int}(X_+)$. Pick an embedded path $\gamma$ from the basepoint of $R$ to the basepoint of $T$ that is transverse to $F$. Define $n(\gamma,T):=|\gamma\cap F|\in\Z$. We claim this number modulo 2 is independent of $\gamma$ and $T$. To see this, suppose $\gamma'$ and $T'$ are two other such choices. The concatenation $\gamma''$ of $\gamma^{-1}$ and $\gamma'$ is an embedded path from $T$ to $T'$. By a regular homotopy relative to the endpoints, move $\gamma''$ to an embedded path $\sigma\subset\partial M$ that is transverse to $L$. We calculate modulo 2
\[
n(\gamma,T)+n(\gamma',T')\equiv |\gamma''\cap F|\equiv |\sigma\cap L|\pmod 2.
\]
To compute the latter, define $\sigma_\pm:=\sigma\cap X_\pm$ and consider the decomposition $\sigma=\sigma_+\cup\sigma_-$. Recall no region of $X_+$ borders another region of $X_+$, so the division of $\sigma$ must alternate from $+$ to $-$. As the division of $\sigma$ starts and ends in a $+$ region, we deduce $|\sigma\cap L|$ is even.
This shows $n(\gamma, T)$ modulo 2 is independent of $\gamma$ and $T$. We denote this quantity by $n(R)\in\{0,1\}$. Define
\[
M_+:=\operatorname{cl}\left(\textstyle{\bigcup_{n(R)=0}R}\right),\qquad M_-:=\operatorname{cl}\left(\textstyle{\bigcup_{n(R)=1}R}\right).
\]
It is clear that $M_+\cap M_-=F$ and $M=M_+\cup_F M_-$.

\end{proof}

With this lemma we are able to prove the following.

\begin{theorem} 
\label{thm:HyperbolicSeifert}
Every $\mu$-coloured doubly slice link admits a doubly isotropic coloured boundary Seifert matrix.
\end{theorem}
\begin{proof}
Let $M_1,\dots,M_\mu\subset S^4$ be disjoint $3$-balls so that, for each $i$, the surface $M_i\cap S^3=F_i\cup S_i\subset S^3$ is such that $\partial F_i=L_i$ and $S_i$ is closed. For each $i$, use Lemma \ref{lem:partitionargument} to obtain a division~$M_i=M_i^+\cup_{F_i} M_i^-$ so that $M_i^+\cap M_i^-=F_i$.
Exactly as in~\cite[proof of Lemma 3.3]{KoSeifert}, we argue that
$B^+_i:=\ker(H_1(F_i)\to H_1(M^+_i))$
is isotropic (the argument for the ``$-$'' version is entirely similar). For each~$i$, choose any splitting~$H_i\cong B^+_i\oplus C_i^+$, where
$H_i:=H_1(F_i)$.
Then if~$\alpha_i\in B_i^+$ and~$\beta_j\in B_j^+$ the value~$A_{ij}(\alpha_i,\beta_j)$ may be computed as follows. Let~$P_i\subset M_i^+$ and~$Q_j\subset M_j^+$ be surfaces bounded by~$\alpha_i$ and~$\beta_j$. The value of~$A_{ij}(\alpha_i,\beta_j)$ is then the algebraic intersection count in~$S^4$ between~$P_i$ and a push-off of~$Q_j$ in the positive normal direction of~$M_j^+$. 
When~$i\neq j$ the fact that~$M_i^+$ is disjoint from $M_j^+$ (as they are subsets of distinct $3$-balls $M_i$ and $M_j$) implies this intersection is empty. When~$i=j$, consider that a positive push-off of $M_i^+$ does not meet $M_i^+$ as the normal bundle is trivial. Hence~$P_i$ does not intersect the push-off of~$Q_j$. Thus we have shown each of $\{B_i^+\}$ and $\{B_i^-\}$ is an isotropic family.

Now for each~$i$, there is a Mayer-Vietoris sequence
\[
0=H_2(M_i)\to H_1(F_i\cup S_i)\to H_1(M^+_i)\oplus H_1(M^-_i)\to H_1(M_i)=0,
\]
 so the inclusions induce isomorphisms
 $H_1(F_i)\cong H_1(M^+_i)\oplus H_1(M^-_i)$. 
Under this isomorphism, we see that now~$B_i^+\cong H_1(M_i^-)$ and~$B_i^-\cong H_1(M_i^+)$, so that $H_1(F_i)\cong B^+_i\oplus B^-_i$. Thus the coloured boundary Seifert matrix $A$, corresponding to $\bigcup_iF_i$, is doubly isotropic.
\end{proof}
\color{black}

Combining Proposition \ref{thm:HyperbolicSeifert}
with Remark~\ref{rem:MetabolicIsMetabolic}, items iii) and iv), we obtain:

\begin{corollary}\label{cor:hyperbolicseifertmatrix}
A strongly doubly slice link admits a boundary Seifert matrix~$A$ such that each matrix~$A_{ii}$ is hyperbolic, and $A$ is hyperbolic as an integral matrix.
\end{corollary}

\subsection{Doubly slice links have doubly isotropic~$C$-complex matrices}
\label{sub:Ccomplex}

In this subsection, we show how {Theorem}~\ref{thm:HyperbolicSeifert} gives information about abelian invariants of doubly slice coloured links, thus recovering a subset of the results from Theorems~\ref{thm:LowerBound4D} and~\ref{thm:BlanchfieldStronglySlice}.
To achieve this, we first recall the notion of a~$C$-complex~\cite{Cooper,CooperThesis} and the 3-dimensional interpretation of the multivariable signature~\cite[Section 2]{CimasoniFlorens}.
\medbreak

\begin{definition}
\label{def:CComplex}
A \emph{$C$-complex} for a~$\mu$-coloured link~$L=L_1\cup\dots\cup L_\mu$ is a union~$F=F_1\cup\dots\cup F_\mu$ of surfaces in~$S^3$ such that:
\begin{enumerate}[leftmargin=*]\setlength\itemsep{0em}
\item$F_i$ is a Seifert surface for the sublink~$L_i$ (possibly disconnected but with no closed components);
\item$F_i\cap F_j$ is either empty or a union of clasps for all~$i\neq j$;
\item$F_i\cap F_j\cap F_k$ is empty for all~$i,j,k$ pairwise distinct.
\end{enumerate}
\end{definition}

Every coloured link admits a~$C$-complex~\cite[Lemma 1]{CimasoniPotential}.
Note that for~$\mu=1$, a~$C$-complex for a~$1$-coloured link~$L$ is a Seifert surface for~$L$.
Furthermore, a boundary Seifert surface for a boundary coloured  link~$L$ (in the sense of Definition~\ref{def:BoundaryLink}) is a~$C$-complex for~$L$.

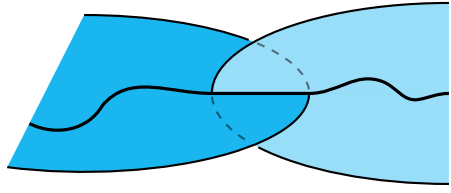
\begin{figure}[h!tb]
\begin{tikzpicture}
\begin{scope}[scale=0.8]
\tikzset{
    partial ellipse/.style args={#1:#2:#3}{
        insert path={+ (#1:#3) arc (#1:#2:#3)}
    }
}

\draw[thick, draw opacity = 0, black, fill=white, fill opacity =1] (3,0) [partial ellipse=180:270:3.9 and 1.5] -- (3,0);
\draw[thick, draw opacity = 0, cyan, fill=cyan, fill opacity =0.4] (3,0) [partial ellipse=180:270:3.9 and 1.5] -- (3,0);
\draw[line width =0.001mm, draw opacity = 0.4, cyan] (0.7,0) -- (3,0);
\draw[ thick, draw opacity = 0, black, fill=white, fill opacity =1] (-3,0) [partial ellipse=-110:90:3.7 and 1.3];
\draw[ thick, draw opacity = 0, black, fill=cyan, fill opacity =0.9] (-3,0) [partial ellipse=-110:90:3.7 and 1.3];
\draw[thick, draw opacity = 0, black, fill=white, fill opacity =1] (3,0) [partial ellipse=90:180:3.9 and 1.5] -- (3,0);
\draw[thick, draw opacity = 0, black, fill=cyan, fill opacity =0.4] (3,0) [partial ellipse=90:180:3.9 and 1.5] -- (3,0);

\draw[thick, black] (3,0) [partial ellipse=90:180.8:3.9 and 1.5];
\draw[ thick, black, dashed, draw opacity = 0.5] (3,0) [partial ellipse=182:213:3.9 and 1.5];
\draw[thick, black] (3,0) [partial ellipse=216.5:270:3.9 and 1.5];

\draw[thick] (-3,0) [partial ellipse=-110:0.9:3.7 and 1.3];
\draw[thick, dashed, draw opacity = 0.5] (-3,0) [partial ellipse=3:40:3.7 and 1.3];
\draw[ thick] (-3,0) [partial ellipse=43:90:3.7 and 1.3];

\draw[very thick, black] (-3.9,-0.5)
to [out=-26, in=-120] (-2.7,-.2) 
to [out=50, in=-180] (-0.9,0) 
to [out=0, in=-180] (0.7,0) 
to [out=0, in=160] (1.9,0.2)
to [out=-20, in=190] (2.5,-0.1)
to [out=10, in=-180] (3,0);
\end{scope}

\end{tikzpicture}
\caption{An arc, forming part of a 1-cycle, running through a clasp intersection.}
\label{fig:clasp}
\end{figure}

Given a~$C$-complex~$F$ for a coloured link~$L$ and a sequence~$\eps=(\eps_1,\dots,\eps_\mu)$ of~$\pm 1$'s, define a map~$i^\eps\colon H_1(F)\to H_1(S^3\setminus F)$ as follows. Any homology class in~$H_1(F)$ can be represented by an oriented cycle~$x$ which behaves as
illustrated in Figure~\ref{fig:clasp} whenever crossing a clasp. 
Define~$i^\eps([x])$ as the class of the~$1$-cycle obtained by pushing~$x$ in the~$\eps_i$-normal direction off~$F_i$ for~$i=1,\dots,\mu$.
The \emph{generalised Seifert form} associated to~$F$ and~$\varepsilon$ is then defined as
\begin{align*}
\alpha^\eps\colon H_1(F)\times H_1(F)&\to\Z \\
\quad(x,y)&\mapsto\ell k(i^\eps(x),y).
\end{align*}
Fix a basis of~$H_1(F)$ and denote by~$A^\eps$ the matrix of~$\alpha^\eps$. 
The resulting~$2^\mu$ matrices are called \emph{generalised Seifert matrices} for the coloured link~$L$.
Note that since~$A^{-\eps}=(A^\eps)^T$ for all~$\eps$, there are only~$2^{\mu-1}$ matrices to calculate in practice.
\begin{definition}
\label{def:CcomplexMatrix}
Given a~$\mu$-coloured link~$L$, a choice of~$C$-complex~$F$, and a choice of basis for~$H_1(F)$, the \emph{$C$-complex matrix} is
$$ H(t_1,\ldots,t_\mu)=\sum_\eps\prod_{i=1}^\mu(1-t_i^{\eps_i})\,A^\eps$$
where the~$A^\varepsilon$ are the generalised Seifert matrices described above.
\end{definition}

When~$\mu=1$, the definition of a~$C$-complex is that of a Seifert surface, and a~$C$-complex matrix is one of the form~$(1-t^{-1})A+(1-t)A^T$, where~$A$ is a Seifert matrix for the oriented link~$L$.

The original definition of the multivariable signature~$\sigma_L \colon \mathbb{T}_*^\mu \to \Z$ was in terms of~$C$-complex matrices~\cite[Subsection 2.2]{CimasoniFlorens}, not the way it is stated in Definition \ref{def:MultivariableSignature1}.
The connection between the definition we gave and the original definition is as follows, and was proved in~\cite[Proposition~1.1]{ConwayNagelToffoli}.

\begin{proposition}
\label{prop:CComplexMatrix}
Let~$H:=H_F(t_1,\ldots,t_\mu)$ be an~$(n \times n)$~$C$-complex matrix coming from the choice of a~$C$-complex~$F$ for a coloured link~$L$. Write~$\beta_0(F)$ for the number of connected components of~$F$. For any~$(\omega_1,\ldots,\omega_\mu) \in \mathbb{T}^\mu_*$, the multivariable signature and nullity of~$L$ can be described as
\begin{align*}
\sigma_L(\omega_1,\ldots,\omega_\mu)&=\operatorname{sign}(H_F(\omega_1,\ldots,\omega_\mu)), \\
\eta_L(\omega_1,\ldots,\omega_\mu)&=\operatorname{null}(H_F(\omega_1,\ldots,\omega_\mu))+\beta_0(F)-1.
\end{align*}
\end{proposition}

The attentive reader will notice that in Theorem~\ref{thm:MainTheoremIntro} from the introduction, the multivariable signature was defined on the whole torus~$\mathbb{T}^\mu$, while in the remainder of the paper, we worked with~$\mathbb{T}^\mu_*$.
We record a quick comment about this fact.
\begin{remark}
\label{rem:VariableEquals1}
If one takes the~$C$-complex approach to multivariable invariants, then the multivariable signature can be defined on the whole of~$\mathbb{T}^\mu$.
Indeed, since~$H(\boldsymbol{\mathbb{\omega}})$ is defined at each~$\boldsymbol{\omega}  \in \mathbb{T}^\mu$, and vanishes on~$\mathbb{T}^\mu  \setminus \mathbb{T}^\mu_*$, it is possible to take as a convention that~$\sigma_L$ vanishes as soon as one~$\omega_i$ vanishes; this is what we implicitly did in the introduction.
Another interpretation of~$\sigma_L$ on~$\mathbb{T}^\mu  \setminus \mathbb{T}^\mu_*$ can be found in~\cite[Remark 3.6]{DegtyarevFlorensLecuonaSplice}, but we will not pursue this thread further.
Finally, note that throughout the article we always think of the multivariable nullity as being defined on~$\mathbb{T}^\mu_*$.
\end{remark}

While it is always possible to use a~$C$-complex matrix to compute the Blanchfield form of a coloured link (see \cite[Theorem~1.1]{Conway}), it is shown in \cite[Theorem 4.6]{Conway} that in the case of boundary links a particularly clean presentation of the Blanchfield form is possible.
We now use ideas from the proof of that theorem to show that doubly slice links admit doubly isotropic~$C$-complex matrices; this will prove Theorem~\ref{thm:SeifertBlanchfieldIntro} from the introduction.

\begin{theorem}
\label{thm:SeifertBlanchfield}
Any doubly slice coloured link admits a {collection of  generalised Seifert matrices~$\lbrace A^\varepsilon \rbrace$ where~$A^\varepsilon$ is doubly isotropic for each $\varepsilon$, and also admits a} doubly isotropic~$C$-complex matrix.
Furthermore, a strongly doubly slice link admits a {collection of  generalised Seifert matrices~$\lbrace A^\varepsilon \rbrace$ where~$A^\varepsilon$ is hyperbolic for each $\varepsilon$, and also admits a} hyperbolic $C$-complex matrix.
\end{theorem}
\begin{proof}
Let~$F$ be a boundary Seifert surface for a~$\mu$-coloured link~$L$, and let~$A=(A_{ij})$ be a coloured boundary Seifert matrix associated to~$F$.
Viewing~$F$ as a~$C$-complex for~$L$, we recall how~$A$ can be used to describe a~$C$-complex matrix for~$L$, generalising the argument for the~$\mu=n$ case from~\cite[Theorem 4.6]{Conway}.

If~$i \neq j$, since~$F$ is a coloured boundary Seifert surface,~$A_{ij}^\varepsilon$ is independent of~$\varepsilon$ and is equal to the block~$A_{ij}$ of the coloured boundary Seifert matrix~$A$. 
Similarly, for each~$\varepsilon$ with~$\varepsilon_i=-1$, the restriction of~$A^\varepsilon$ to~$H_1(F_i) \times H_1(F_i)$ is equal to the block~$A_{ii}$ (for~$\varepsilon_i$, it equals~$A_{ii}^T$). 
Combining Theorem~\ref{thm:HyperbolicSeifert} and Remark~\ref{rem:MetabolicIsMetabolic}, it follows that~$A^\varepsilon$ is doubly isotropic for each~$\varepsilon$.
The additional statement for~$\mu=n$ follows as in Corollary~\ref{cor:hyperbolicseifertmatrix}.

We now prove the assertions about the $C$-complex matrices.
Let~$H_i=(1-t_i)A_{ii}^T+(1-t_i^{-1})A_{ii}$ denote the corresponding~$C$-complex matrix for the sublink~$L_i$ and let~$u$ denote~$\prod_{j=1}^\mu(1-t_j)$.
Using Definition~\ref{def:CcomplexMatrix} and arguing as in~\cite[Proof of Theorem 4.6]{Conway}, we see that the~$C$-complex matrix~$H$ associated to~$F$ is 
$$
\begin{pmatrix}
u\overline{u}(1-t_1)^{-1}(1-t_1^{-1})^{-1}H_1 &u\overline{u} A_{12} & \ldots & u\overline{u} A_{1{\mu}} \\
u\overline{u} A_{21} & u\overline{u}(1-t_2)^{-1}(1-t_2^{-1})^{-1}H_2 & \ldots & u\overline{u} A_{2{\mu}} \\
\vdots & \ddots & \ddots& \vdots \\
u\overline{u} A_{{\mu}1} & u\overline{u} A_{{\mu}2} & \ldots & u\overline{u}(1-t_{{\mu}})^{-1}(1-t_{{\mu}}^{-1})^{-1}H_{{\mu}}
\end{pmatrix}.$$
Since~$L$ is doubly slice, each~$A_{ij}$ is doubly isotropic by {Theorem}~\ref{thm:HyperbolicSeifert}.
Viewing the~$A_{ij}$ as pairings~$A_{ij} \colon H_i \times H_j \to \Z$, this means that there are submodules~$G_i^{\pm}$ for~$i=1,\ldots,\mu$ such that~$H_i=G_i^- \oplus G_i^+$ and~$A_{ij}$ vanishes on~$G_i \times G_j$.
Arguing as in the second item of Remark~\ref{rem:MetabolicIsMetabolic}, this implies that~$H$ is itself doubly isotropic: the two isotropic submodules are given by~$\oplus_{i=1}^\mu G_i^{-}$ and~$\oplus_{i=1}^\mu G_i^{+}$.
In the case that~$\mu=n$, this implies that~$H$ is hyperbolic; recall the fourth item of Remark~\ref{rem:MetabolicIsMetabolic}.
\end{proof}

As a corollary, we can give alternative proofs of our multivariable obstructions to a link being doubly slice, first stated in Theorems~\ref{thm:LowerBound4D} and~\ref{thm:BlanchfieldStronglySlice}.
\begin{corollary}\label{cor:seifertway}
Let~$L$ be a~$\mu$-coloured link.
\begin{enumerate}[leftmargin=*]\setlength\itemsep{0em}
\item If~$L$ is doubly slice, then~$\sigma_L({\boldsymbol{\omega}})=0$ for all~${\boldsymbol{\omega}}\in \mathbb{T}_*^\mu$.
\item If~$L$ is strongly doubly slice, then~$\Bl_L$ is hyperbolic.
\end{enumerate}
\end{corollary}
\begin{proof}
We recalled in the first item of Proposition~\ref{prop:CComplexMatrix} that~$\sigma_L(\omega_1,\ldots,\omega_\mu)=\sign H(\omega_1,\ldots,\omega_\mu)$ for any~$C$-complex matrix~$H$.
Since~$L$ is doubly slice, it admits a doubly isotropic~$C$-complex matrix (by Theorem~\ref{thm:SeifertBlanchfield}), and the first statement therefore 
reduces to the known fact that the signature of a doubly isotropic complex Hermitian matrix vanishes. We verify this algebraic fact below in Lemma~\ref{lem:peterfeller}.

We prove the second statement.
Assume that~$L$ is strongly doubly slice, so that~$L$ is a boundary link.
Looking at~\cite[proof of Theorem 4.6]{Conway}, we obtain the following: for~$\mu=n$, the~$C$-complex matrix~$H$ described during the proof of Theorem~\ref{thm:SeifertBlanchfield} {(whose size we denote $m$)} is invertible over~$Q$, presents~$TH_1(X_L;\Lambda_S)$ and~$\Bl_L$ is isometric to the pairing 
\[
\begin{array}{rcl}
\lambda_H \colon \Lambda_S^{{m}}/H^T\Lambda_S^{{m}} \times  \Lambda_S^{{m}}/H^T\Lambda_S^{{m}} &\to &Q/\Lambda_S \\ 
([x],[y]) &\mapsto &y^T H^{-1} \overline{x}.
\end{array}
\]
The statement therefore reduces to proving the following algebraic statement: if~$H$ is a hyperbolic matrix over~$\Lambda_S$ with non-zero determinant, then~$\lambda_H$ is a hyperbolic linking form. The proof of this is standard in the theory of linking forms; for instance the metabolisers are obtained as in~\cite[Proposition C.1]{FriedlThesis} and they will clearly be complementary direct summands. This concludes the proof of the corollary.
\end{proof}

It remains to prove the following algebraic lemma, which was used in the proof of Corollary~\ref{cor:seifertway}.

\begin{lemma}\label{lem:peterfeller} Every doubly isotropic complex Hermitian matrix has vanishing signature.
\end{lemma}

\begin{proof} Let~$A$ be such a matrix and let~$G^-$ and~$G^+$ be the complementary isotropic submodules. Choosing bases for~$G^-$ and~$G^+$, there is a block decomposition
\[
A\sim\left(\begin{matrix} 0 & \overline{C} \\ C^T & 0\end{matrix}\right)
\]
where~$C$ is a complex matrix and~$\overline{C}$ denotes the complex conjugate. We assume the bases are ordered so that~$\overline{C}$ determines a linear map~$G^-\to G^+$. Choose a new basis for~$G^-$ extending a basis for~$\ker(\overline{C})\subset G^-$ and choose a new basis for~$G^+$ extending a basis for~$\im(\overline{C})$. Order the new bases so that with respect to them this linear map is
\[
\left(\begin{matrix} \overline{D} & 0 \\ 0& 0\end{matrix}\right)\colon G^-\to G^+
\]
where~$\overline{D}$ is invertible.
In fact, by an additional congruence, $D$ can be assumed to be the identity matrix.
 We thus have congruences
\[
A\sim\left(\begin{matrix} 
0 & 0 & \overline{D} & 0 \\
0 & 0 & 0 & 0 \\
D^T & 0 & 0 & 0\\
0 & 0 & 0 & 0 \\
\end{matrix}\right)
\sim
\left(\begin{matrix} 0 & \overline{D} \\ D^T & 0\end{matrix}\right)\oplus \left(\begin{matrix} 0 & 0 \\ 0 & 0\end{matrix}\right),
\]
and the latter clearly has vanishing signature.
\end{proof}

\section{Examples}
\label{sec:Examples}

In this section we will present some applications of the obstructions developed in this article. We introduce some useful terminology. 

\begin{definition}A \emph{quasi-orientation} for an unoriented link~$L$ is an equivalence class of orientations for~$L$, where two orientations are equivalent if they differ by reversing the orientation on each component of the link.
\end{definition}

If a link~$L$ is weakly doubly slice with a given orientation then the link with the reversed orientation is also weakly doubly slice (this is easily seen by reversing the orientation for the unknotted~$2$-sphere of which~$L$ is a cross-section). So given an unoriented link, the task of checking which orientations correspond to weak double slicings reduces to checking one orientation per quasi-orientation class.

\begin{remark} There is a coloured version of this concept that we won't utilise, but we mention for completeness. Given an ordered partition of a link~$L=L_1\cup\dots\cup L_\mu$, one could define a \emph{coloured quasi-orientation} with respect to the ordered partition to be an equivalence class of orientations for~$L$ where equivalence is the transitive closure of the relation determined by an overall orientation reversal on~$L_j$, for some~$1\leq j\leq \mu$. If a~$\mu$-coloured link~$L$ is doubly slice then it is doubly slice with respect to the colouring determined by any orientation in the coloured quasi-orientation class.
\end{remark}

For the convenience of the reader, we summarise our obstructions:
a strongly doubly slice link must  have~$\sigma_L \equiv 0$ (which implies that~$\sigma_L^{LT} \equiv 0$), Alexander nullity~$\beta(L)=\beta_0(L)-1$ (which implies that the multivariable Alexander polynomial vanishes identically:~$\Delta_L \equiv 0$ for~$\beta_0(L)>1$) as well as~$TH_1(X_L;\Lambda_S)=G_1 \oplus G_2$ with~$\operatorname{Ord}(G_1)=\overline{\operatorname{Ord}(G_2)}$; it must also be a slice boundary link with doubly slice components. In particular, a
strongly doubly slice link must have vanishing linking matrix.
A weakly doubly slice link must have~$TH_1(\Sigma_2(L))=G \oplus G$ for some abelian group~$G$ (thus the link determinant must be a square number), vanishing Levine-Tristram signature and~$TH_1(X_L;\Q[t^{\pm 1},(1-t)^{-1}])=G \oplus \overline{G}$ for some~$\Q[t^{\pm 1},(1-t)^{-1}]$-module~$G$.

On the constructive side, we recall a geometric operation originally described in~\cite[Lemma 4.9.2]{IssaThesis}. An \emph{$(a,b)$-tangle}~$T$ is a collection of properly embedded oriented arcs in~$D^2\times[0,1]$ with boundary consisting of~$a$ points in~$D^2\times\{0\}$ and~$b$ points in~$D^2\times\{1\}$. Suppose~$D^2\times[0,1]\subset S^3$ intersects a link~$L$ in an~$(a,b)$-tangle~$T$ and consider the tangle replacement operation depicted in Figure~\ref{fig:folding}.
  \begin{figure}
  \begin{center}
    \begin{tikzpicture}
    \node[inner sep=0pt] at (0,0)
    {\includegraphics[width=0.9\textwidth]{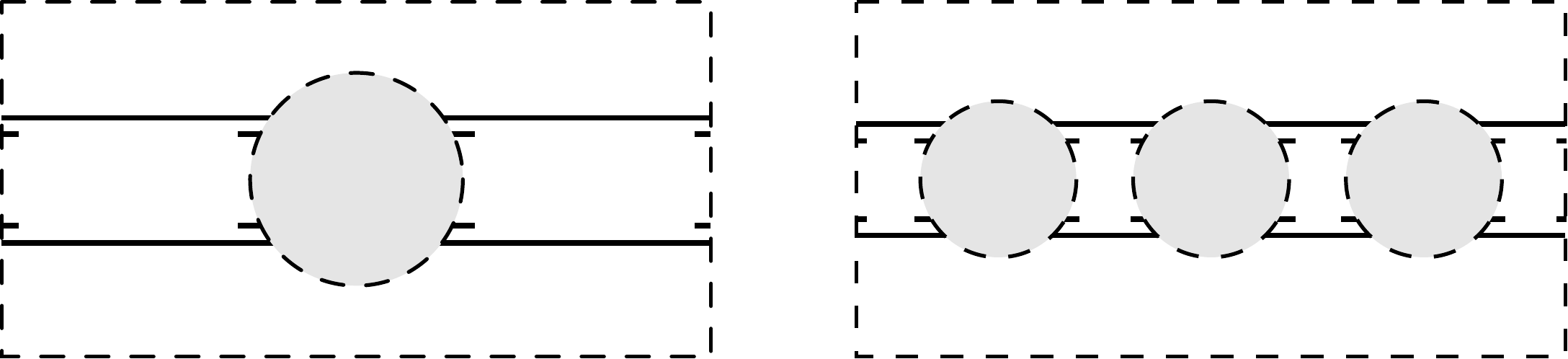}};
    
    \node[inner sep=0pt] at (-6.35,0.2) {\scalebox{1.8}{$\vdots$}};
    \node[inner sep=0pt] at (-2.4,0.2) {\scalebox{1.8}{$\vdots$}};
    \node[inner sep=0pt] at (-4.9,0.2) {\scalebox{1.8}{$\vdots$}};
    \node[inner sep=0pt] at (-0.95,0.2) {\scalebox{1.8}{$\vdots$}};
        \node[inner sep=0pt] at (-3.65,0) {\scalebox{1.5}{$T$}};
                \node[inner sep=0pt] at (-6.95,0) {\scalebox{1}{$a$}};
                        \node[inner sep=0pt] at (-0.35,0) {\scalebox{1}{$b$}};
        
            \node[inner sep=0pt] at (0.9,0.15) {\scalebox{1.5}{$\vdots$}};
    \node[inner sep=0pt] at (2.75,0.15) {\scalebox{1.5}{$\vdots$}};
        \node[inner sep=0pt] at (4.58,0.15) {\scalebox{1.5}{$\vdots$}};
          \node[inner sep=0pt] at (6.45,0.15) {\scalebox{1.5}{$\vdots$}};
        \node[inner sep=0pt] at (1.86,0) {\scalebox{1.2}{$T$}};
         \node[inner sep=0pt] at (3.7,0.05) {\scalebox{1.2}{$\widetilde{T}$}};
        \node[inner sep=0pt] at (5.5,0) {\scalebox{1.2}{$T$}};
             \node[inner sep=0pt] at (0.9,-0.75) {\scalebox{1}{$a$}};
                        \node[inner sep=0pt] at (2.75,-0.75) {\scalebox{1}{$b$}};
                             \node[inner sep=0pt] at (4.58,-0.75) {\scalebox{1}{$a$}};
                        \node[inner sep=0pt] at (6.45,-0.75) {\scalebox{1}{$b$}};
        
    \end{tikzpicture}
    \end{center}
  \caption{\label{fig:folding} Issa's folding construction. Left: A dashed rectangle representing~$D^2\times[0,1]$ and containing the~$(a,b)$-tangle~$T$. Right: A dashed rectangle representing~$D^2\times[0,1]$ and containing the replacement~$(a,b)$-tangle. Here,~$\widetilde{T}$ is the effect of reflecting~$T$ in a vertical plane, then changing all the crossings and reversing the orientation.}
  \end{figure}
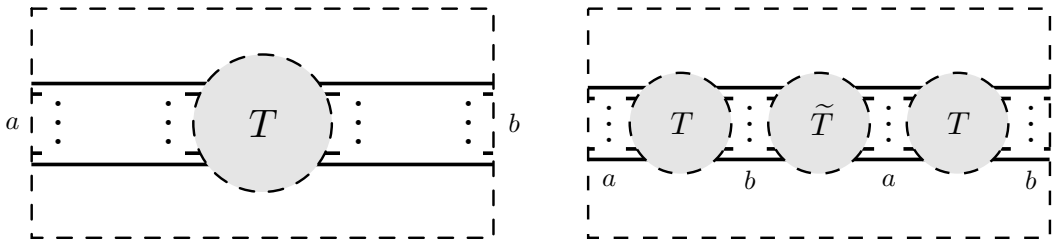
  Observe that if~$L$ is a~$\mu$-coloured link then the effect of the tangle replacement operation inherits a natural orientation and~$\mu$-colouring; cf.~Figure~\ref{fig:nm}. The following statement is moreover true.

\begin{lemma}\label{lem:folding}Suppose~$L$ is a~$\mu$-coloured doubly slice link and~$D^2\times[0,1]\subset S^3$ intersects~$L$ in an~$(a,b)$-tangle~$T$. Then the effect of the tangle replacement operation depicted in Figure \ref{fig:folding} is also a~$\mu$-coloured doubly slice link.
\end{lemma}

\begin{proof} This is an immediate corollary of \cite[Lemma 2.1]{McCoyMcDonald}.
\end{proof}

We will refer to the tangle replacement operation depicted in Figure \ref{fig:folding} as \emph{Issa's folding construction} for reasons that become clear if one reads the proof of \cite[Lemma 2.1]{McCoyMcDonald} (which we do not reproduce here). We present a basic application of the folding construction, for later use.

\begin{example}\label{ex:folding}Consider the pretzel link~$P(a_1,a_2,\dots,a_k)$. For any~$1\leq j\leq k$, by choosing~$T$ to be the~$(2,2)$-tangle described by isolating the~$a_j$ half-twists region and performing Issa's folding construction we obtain~$P(a_1,\dots,a_{j-1},a_j,-a_j,a_j,a_{j+1},\dots a_k)$. This immediately provides for~$\mu$-double sliceness of various pretzel links via Lemma \ref{lem:folding}. For example,~$P(a)$ is the unknot for all~$a\in\Z$ and hence, by iterating the observation above, we have that~$P(a,-a,a,-a,\dots, a)$ is weakly doubly slice for all~$a$, with respect to the unique quasi-orientation class inheritable from~$P(a)$ under the folding construction. It is not necessarily true that other quasi-orientations of these pretzel links are weakly doubly slice; see e.g.~the link L6n1 in Section~\ref{sec:weakly} below. Similarly,~$P(a,-a)$ is the two-component unlink for all~$a\in\Z$ and hence~$P(a,-a,a,-a,\dots, a, -a)$ is~$2$-doubly slice with respect to any orientation and~$2$-colouring inheritable from the~$2$-coloured unlink (hence this link is also weakly doubly slice with respect to these orientations). Both of these facts were partially derived by Donald using a different method~\cite[Corollary 2.7]{Donald}.
\end{example}

\subsection{Links that are~$m$-doubly slice but not~$(m+1)$-doubly slice.}

\begin{example}
\label{ex:12Butnot34}
Consider the~$4$-component link~$L$ in Figure \ref{fig:nm} (right), obtained by applying Issa's folding construction to the tangle in the link in Figure \ref{fig:nm} (left). As the~$2$-component 2-coloured unlink is doubly slice, so~$L$ is~$\mu$-doubly slice for~$\mu=1,2$.
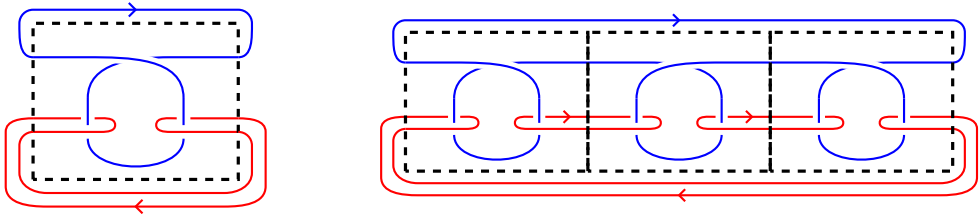
\begin{figure}[h!tb]
\begin{tikzpicture}

\begin{scope}[scale=0.9, shift={(-1,-0.2)}]

\draw[thick, blue] (0.8,1.3)
to [out=-90, in=-90] (2.2,1.3);

\draw[thick, blue] (3,2.5)
to [out=180, in=0] (2,2.5)
to [out=180, in=90] (0.8,1.9);
\draw[thick, blue] (0.8,1.9)
to [out=-90, in=90] (0.8,1.65)
to [out=-90, in=90] (0.8,1.5);
\draw[thick, blue] (2.2,1.9)
to [out=-90, in=90] (2.2,1.65)
to [out=-90, in=90] (2.2,1.5);
\draw[line width=1.5mm, white] (0,2.5)
to [out=0, in=180] (0.8,2.5)
to [out=0, in=90] (2.2,1.9);
\draw[thick, blue] (0,2.5)
to [out=0, in=180] (0.8,2.5)
to [out=0, in=90] (2.2,1.9);

\draw[thick, red] (0,1.4)
to [out=0, in=180] (1,1.4)
to [out=0, in=-90] (1.2,1.5)
to [out=90, in=0] (0.9,1.6);
\draw[thick, red] (0.7,1.6)
to [out=180, in=0] (0,1.6);

\draw[thick, red] (3,1.4)
to [out=180, in=0] (2,1.4)
to [out=180, in=-90] (1.8,1.5)
to [out=90, in=180] (2.1,1.6);
\draw[thick, red] (2.3,1.6)
to [out=0, in=180] (3,1.6);

\draw[very thick, black, dashed] (0,0.7) rectangle (3,3);

\draw[thick, blue] (0,2.5)
to [out=180, in=-90] (-0.2,3)
to [out=90, in=180] (0,3.2)
to [out=0, in=180] (3,3.2)
to [out=0, in=90] (3.2,3)
to [out=-90, in=0] (3,2.5);
\draw[thick, blue] (1.4,3.3) -- (1.5,3.2) -- (1.4,3.1);

\draw[thick, red] (3, 1.4)
to [out=0, in=90] (3.2, 1.2)
to [out=-90, in=90] (3.2, 0.8)
to [out=-90, in=0] (2.8, 0.5)
to [out=180, in=0] (0.2, 0.5)
to [out=180, in=-90] (-0.2, 0.8)
to [out=90, in=-90] (-0.2, 1.2)
to [out=90, in=180] (0, 1.4);

\draw[thick, red] (3, 1.6)
to [out=0, in=90] (3.4, 1.4)
to [out=-90, in=90] (3.4, 0.6)
to [out=-90, in=0] (2.8, 0.3)
to [out=180, in=0] (0.2, 0.3)
to [out=180, in=-90] (-0.4, 0.6)
to [out=90, in=-90] (-0.4, 1.4)
to [out=90, in=180] (0, 1.6);
\draw[thick, red] (1.6,0.2) -- (1.5,0.3) -- (1.6,0.4);

\end{scope}

\begin{scope}[scale=0.8, shift={(5,0)}]

\draw[thick, blue] (0.8,1.3)
to [out=-90, in=-90] (2.2,1.3);

\draw[thick, blue] (3,2.5)
to [out=180, in=0] (2,2.5)
to [out=180, in=90] (0.8,1.9);
\draw[thick, blue] (0.8,1.9)
to [out=-90, in=90] (0.8,1.65)
to [out=-90, in=90] (0.8,1.5);
\draw[thick, blue] (2.2,1.9)
to [out=-90, in=90] (2.2,1.65)
to [out=-90, in=90] (2.2,1.5);
\draw[line width=1.5mm, white] (0,2.5)
to [out=0, in=180] (0.8,2.5)
to [out=0, in=90] (2.2,1.9);
\draw[thick, blue] (0,2.5)
to [out=0, in=180] (0.8,2.5)
to [out=0, in=90] (2.2,1.9);

\draw[thick, red] (0,1.4)
to [out=0, in=180] (1,1.4)
to [out=0, in=-90] (1.2,1.5)
to [out=90, in=0] (0.9,1.6);
\draw[thick, red] (0.7,1.6)
to [out=180, in=0] (0,1.6);

\draw[thick, red] (3,1.4)
to [out=180, in=0] (2,1.4)
to [out=180, in=-90] (1.8,1.5)
to [out=90, in=180] (2.1,1.6);
\draw[thick, red] (2.3,1.6)
to [out=0, in=180] (3,1.6);

\draw[thick, red] (2.6,1.7) -- (2.7,1.6) -- (2.6,1.5);

\draw[thick, blue] (0,2.5)
to [out=180, in=-90] (-0.2,3)
to [out=90, in=180] (0,3.2)
to [out=0, in=180] (9,3.2)
to [out=0, in=90] (9.2,3)
to [out=-90, in=0] (9,2.5);

\draw[thick, red] (9, 1.4)
to [out=0, in=90] (9.2, 1.2)
to [out=-90, in=90] (9.2, 0.8)
to [out=-90, in=0] (8.8, 0.5)
to [out=180, in=0] (0.2, 0.5)
to [out=180, in=-90] (-0.2, 0.8)
to [out=90, in=-90] (-0.2, 1.2)
to [out=90, in=180] (0, 1.4);

\draw[thick, red] (9, 1.6)
to [out=0, in=90] (9.4, 1.4)
to [out=-90, in=90] (9.4, 0.6)
to [out=-90, in=0] (8.8, 0.3)
to [out=180, in=0] (0.2, 0.3)
to [out=180, in=-90] (-0.4, 0.6)
to [out=90, in=-90] (-0.4, 1.4)
to [out=90, in=180] (0, 1.6);

\end{scope}

\begin{scope}[xscale=-1, scale=0.8, shift={(-11,0)}]

\draw[thick, blue] (0.8,1.3)
to [out=-90, in=-90] (2.2,1.3);

\draw[thick, blue] (3,2.5)
to [out=180, in=0] (2,2.5)
to [out=180, in=90] (0.8,1.9);
\draw[thick, blue] (0.8,1.9)
to [out=-90, in=90] (0.8,1.65)
to [out=-90, in=90] (0.8,1.5);
\draw[thick, blue] (2.2,1.9)
to [out=-90, in=90] (2.2,1.65)
to [out=-90, in=90] (2.2,1.5);
\draw[line width=1.5mm, white] (0,2.5)
to [out=0, in=180] (0.8,2.5)
to [out=0, in=90] (2.2,1.9);
\draw[thick, blue] (0,2.5)
to [out=0, in=180] (0.8,2.5)
to [out=0, in=90] (2.2,1.9);

\draw[thick, red] (0,1.4)
to [out=0, in=180] (1,1.4)
to [out=0, in=-90] (1.2,1.5)
to [out=90, in=0] (0.9,1.6);
\draw[thick, red] (0.7,1.6)
to [out=180, in=0] (0,1.6);

\draw[thick, red] (3,1.4)
to [out=180, in=0] (2,1.4)
to [out=180, in=-90] (1.8,1.5)
to [out=90, in=180] (2.1,1.6);
\draw[thick, red] (2.3,1.6)
to [out=0, in=180] (3,1.6);

\end{scope}

\begin{scope}[xscale=1, scale=0.8, shift={(8,0)}]

\draw[thick, red] (2.6,1.7) -- (2.7,1.6) -- (2.6,1.5);

\draw[thick, blue] (1.4,3.3) -- (1.5,3.2) -- (1.4,3.1);
\draw[thick, red] (1.6,0.2) -- (1.5,0.3) -- (1.6,0.4);

\end{scope}

\begin{scope}[scale=0.8, shift={(11,0)}]

\draw[thick, blue] (0.8,1.3)
to [out=-90, in=-90] (2.2,1.3);

\draw[thick, blue] (3,2.5)
to [out=180, in=0] (2,2.5)
to [out=180, in=90] (0.8,1.9);
\draw[thick, blue] (0.8,1.9)
to [out=-90, in=90] (0.8,1.65)
to [out=-90, in=90] (0.8,1.5);
\draw[thick, blue] (2.2,1.9)
to [out=-90, in=90] (2.2,1.65)
to [out=-90, in=90] (2.2,1.5);
\draw[line width=1.5mm, white] (0,2.5)
to [out=0, in=180] (0.8,2.5)
to [out=0, in=90] (2.2,1.9);
\draw[thick, blue] (0,2.5)
to [out=0, in=180] (0.8,2.5)
to [out=0, in=90] (2.2,1.9);

\draw[thick, red] (0,1.4)
to [out=0, in=180] (1,1.4)
to [out=0, in=-90] (1.2,1.5)
to [out=90, in=0] (0.9,1.6);
\draw[thick, red] (0.7,1.6)
to [out=180, in=0] (0,1.6);

\draw[thick, red] (3,1.4)
to [out=180, in=0] (2,1.4)
to [out=180, in=-90] (1.8,1.5)
to [out=90, in=180] (2.1,1.6);
\draw[thick, red] (2.3,1.6)
to [out=0, in=180] (3,1.6);

\end{scope}

\begin{scope}[scale=0.8, shift={(5,0)}]
\draw[very thick, black, dashed] (0,0.7) rectangle (3,3);
\end{scope}
\begin{scope}[scale=0.8, shift={(8,0)}]
\draw[very thick, black, dashed] (0,0.7) rectangle (3,3);
\end{scope}
\begin{scope}[scale=0.8, shift={(11,0)}]
\draw[very thick, black, dashed] (0,0.7) rectangle (3,3);
\end{scope}

\end{tikzpicture}
\caption{The~$2$-coloured link~$L$ (right) for Example \ref{ex:12Butnot34} built using Issa's folding construction on the boxed tangle in the~$2$-coloured unlink (left).}
\label{fig:nm}
\end{figure}
Since~$L$ has vanishing pairwise linking numbers, its strong sliceness cannot immediately be ruled out using this easy check.
However a computation shows that its Alexander nullity equals~$1$, ruling out~$\mu$-double sliceness for~$\mu=3,4$; recall Proposition~\ref{prop:AlexanderNullity}.
In fact this nullity obstruction shows that~$L$ is neither strongly slice nor a boundary link.
\end{example}

\begin{example}
\label{ex:L10n36}
Consider the link~$L=\text{L10n36}$. It is a 2-component link with one unknotted component and component whose knot type is~$3_1\#-3_1$, both of which are doubly slice and do not link each other.
One can show that~$L$ is isotopic to the pretzel link~$P(2,-2,3,-3)$ which was proved to be weakly doubly slice by Donald~\cite[Proposition 2.10]{Donald}; in particular Corollary~\ref{cor:WeaklyDoublySlice} ensures that~$\sigma_L^{LT} \equiv 0$.
A~$C$-complex computation shows that~$H_1(X_L;\Lambda_S)=\Lambda_S$ so that~$\Delta_L=0,\Delta_L^{tor}=1$.
Since~$\sigma_L^{LT} \equiv 0$, it follows from~\cite[Theorem 4.1]{CimasoniFlorens} that~$\sigma_L\equiv 0$ and therefore all our obstructions to being strongly doubly slice are inconclusive.
However~$L$ is not strongly doubly slice since it is not a boundary link~\cite[Section 6]{Crowell}.
\end{example}

\subsection{Strongly doubly slice links with up to 11 crossings}

In order to detect potential strongly doubly slice links, we first use LinkInfo~\cite{LinkInfo} to list low crossing links with vanishing Murasugi signature~$\sigma^{LT}_L(-1)$, linking matrix and Alexander polynomial; the Murasugi nullity~$\eta_L(-1)$ must also be at least~$b_0(L)-1$ (the values of the Levine-Tristram signature and nullity at~$\omega \neq -1$ are not listed).

\begin{proposition}
\label{prop:StronglyExamples}
Among all prime oriented non-trivial links with~$11$ or fewer crossings, there are no strongly doubly slice links.
\end{proposition}
\begin{proof}
Among all oriented prime links with~$11$ or fewer crossings there are only three with vanishing Murasugi signature, linking matrix and Alexander polynomial and whose Murasugi nullity is at least~$b_0(L)-1$: the 2-component links L10n32, L10n36 and L11n247; each time for all orientations.

However, L10n32 has a non doubly slice knot as a component (a Stevedore's knot~$6_{1}$) so it cannot be strongly doubly slice,~$L=\text{L11n247}$ has~$TH_1(\Sigma_2(L))=\Z_{17}$ so it is not even weakly doubly slice (by Proposition~\ref{prop:LinkingFormIntro}), and we already mentioned in Example~\ref{ex:L10n36} that the link L10n36 is not strongly doubly slice.
All three of these arguments hold regardless of the orientations.
\end{proof}

\subsection{Weakly doubly slice links with up to 9 crossings}\label{sec:weakly}

We now consider the weakly doubly slice status of links of 9 or fewer crossings. We determine the status for all such links except three 9-crossing links that stubbornly resist our efforts:

\begin{question}\label{q:dontknow}
Using the notation of LinkInfo, are the following links weakly doubly slice?
\[
\text{L9a53 (all orientations)},\,\, \text{L9n21}\{0,0\},\,\, \text{L9n21}\{1,0\},\,\, \text{L9n21}\{1,1\},\,\, \text{L9n25 (all orientations)}.
\]
\end{question}
The vast majority of links with 9 or fewer crossings are seen to be not weakly doubly slice using our most basic abelian invariants. For the handful of cases for which this is not true we either use ad hoc arguments to show they are weakly doubly slice or apply linking numbers to show they are not. The case of the Borromean rings is particularly interesting and for this we use an entirely different argument to show it is not weakly doubly slice.

First, we list the weakly doubly slice links we found.

\begin{proposition}\label{prop:weaklyexamples} Using the notation of LinkInfo, the following quasi-oriented links are weakly doubly slice:
\[
\begin{array}{llllll}
\text{L6n1}\{0,0\}, &\text{L6n1}\{1,0\}, &\text{L6n1}\{1,1\}, &&&\\
\text{L8n8}\{0,0,0\}, &\text{L8n8}\{1,0,0\}, &\text{L8n8}\{0,1,0\}, &\text{L8n8}\{1,1,0\}, &\text{L8n8}\{0,0,1\},& \text{L8n8}\{1,1,1\}.\\
\end{array}
\]
\end{proposition}

\begin{proof}Up to isotopy and reordering of the components the listed quasi-orientation for L6n1 coincide. In fact~$\text{L6n1}\{0,0\}$ is the pretzel link~$P(-2,2,-2)$ with the orientation inherited from Issa's folding construction applied to the unknot as in Example \ref{ex:folding}, so this link is weakly doubly slice.

Similarly, up to isotopy and reordering of the components, the listed quasi-orientations L8n8 coincide. In fact~$\text{L8n8}\{0,0,0\}$ is the pretzel link~$P(-2,2,-2, 2)$ with an orientation inheritable from Issa's folding construction applied to the 2-component unlink as in Example \ref{ex:folding}, so this link is weakly doubly slice.
\end{proof}

We will now obstruct all remaining links from being weakly doubly slice. We record the following observation about linking numbers; cf.~\cite[Lemma~3.1]{McCoyMcDonald}.

\begin{lemma}\label{lem:linkingnumbers}Let~$L$ be a link that is an equatorial cross section of some surface $\Sigma \hookrightarrow S^4$.
\begin{enumerate}[leftmargin=*]\setlength\itemsep{0em}
\item If~$L$ has two components, then~$\ell k(L_1,L_2)=0$.
\item If~$L$ has 3 components, then for some ordering~$\ell k(L_1,L_2)=-\ell k(L_1,L_3)=\ell k(L_2,L_3)$.
\end{enumerate}
\end{lemma}

\begin{proof}Denote by~$F$ and~$G$ the two surfaces into which the knotted~$\Sigma$ is divided by the equatorial~$S^3\subset S^4$.

For the first item, consider that one of the surfaces~$F$ or~$G$ must be connected.
Hence the other surface
has two disjoint components which implies~$\ell k(L_1,L_2)=0$.

For the second item, we must consider two possible configurations for~$\{F,G\}$ depicted in Figure~\ref{fig:wiggle}. 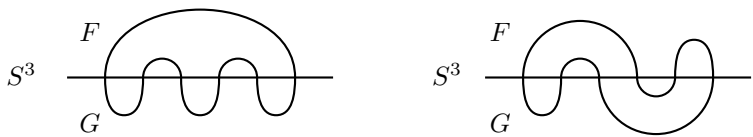
\begin{figure}[h!tb]
\begin{tikzpicture}

\draw[thick, black] (-5,0) -- (-1.5,0);

\draw[thick, black] (-4.5,0)
to [out=up, in=left] (-3.25,0.9)
to [out=right, in=up] (-2,0);

\draw[thick, black] (-4,0)
to [out=up, in=left] (-3.75,0.25)
to [out=right, in=up] (-3.5,0);

\draw[thick, black] (-3,0)
to [out=up, in=left] (-2.75,0.25)
to [out=right, in=up] (-2.5,0);

\draw[thick, black] (-4.5,0)
to [out=down, in=left] (-4.25,-0.5)
to [out=right, in=down] (-4,0);

\draw[thick, black] (-3.5,0)
to [out=down, in=left] (-3.25,-0.5)
to [out=right, in=down] (-3,0);

\draw[thick, black] (-2.5,0)
to [out=down, in=left] (-2.25,-0.5)
to [out=right, in=down] (-2,0);

\node[inner sep=0pt] at (-5.6,0.05) {$S^3$};
\node[inner sep=0pt] at (-4.7,0.6) {$F$};
\node[inner sep=0pt] at (-4.7,-0.6) {$G$};

\begin{scope}[shift={(-1,0)}]
\draw[thick, black] (5,0) -- (1.5,0);

\draw[thick, black] (3.5,0)
to [out=up, in=right] (2.75,0.75)
to [out=left, in=up] (2,0);

\draw[thick, black] (4,0)
to [out=down, in=right] (3.75,-0.25)
to [out=left, in=down] (3.5,0);

\draw[thick, black] (3,0)
to [out=up, in=right] (2.75,0.25)
to [out=left, in=up] (2.5,0);

\draw[thick, black] (4.5,0)
to [out=down, in=right] (3.75,-0.75)
to [out=left, in=down] (3,0);

\draw[thick, black] (4.5,0)
to [out=up, in=right] (4.25,0.5)
to [out=left, in=up] (4,0);

\draw[thick, black] (2.5,0)
to [out=down, in=right] (2.25,-0.5)
to [out=left, in=down] (2,0);

\node[inner sep=0pt] at (1,0.05) {$S^3$};
\node[inner sep=0pt] at (1.7,0.6) {$F$};
\node[inner sep=0pt] at (1.7,-0.6) {$G$};
\end{scope}

\end{tikzpicture}
\caption{The two configurations for a~$2$-sphere in~$S^4$ meeting an equatorial~$S^3$ in a 3-component link.}
\label{fig:wiggle}
\end{figure}
In the left-hand image we suppose one surface,~$F$ say, is connected. This implies~$G$ must consist of 3 disjoint surfaces
and again this implies all pairwise linking must vanish, so the claimed statement follows. Now suppose neither of~$F$ or~$G$ is connected. Then each of~$F$ and~$G$ consists of the disjoint union of a surface with two boundary components and of a surface with a single boundary component,
as depicted in Figure~\ref{fig:wiggle} (right). With the ordering depicted, we now see that
\[\begin{array}{rcccl}
0&=&\ell k(L_1, L_2\cup L_3)&= & \ell k(L_1,L_2) +\ell k(L_1,L_3),\\
0&=&\ell k(L_1\cup L_2,L_3)&= & \ell k(L_1,L_3) +\ell k(L_2,L_3).
\end{array}
\]
Rearranging, the result follows.
\end{proof}

\begin{proposition}\label{prop:weaklystatus}
A link with 9 or fewer crossings is one of the links listed in Proposition \ref{prop:weaklyexamples}, the Borromean rings (with some orientation), one of the links listed in Question \ref{q:dontknow}, or is not weakly doubly slice.
\end{proposition}

\begin{proof}A weakly doubly slice link has vanishing Murasugi signature and the determinant must be a square number. Aside from the links listed in Proposition \ref{prop:weaklyexamples}, Borromean rings and the links listed in Question \ref{q:dontknow}, the following are the only links with 9 or fewer crossings, square determinant and vanishing Murasugi signature, according to LinkInfo:
\[
\begin{array}{llllll}
\text{L8a19}\{0,0\}, &\text{L8a19}\{1,1\}, &\text{L8n3}\{1,0\}, &\text{L8n3}\{0,1\},&&\\
\text{L9a45}\{0,0\}, &\text{L9a45}\{1,0\}, &\text{L9a45}\{0,1\}, &&&\\
\text{L9a46}\{0,0\},&\text{L9a46}\{1,1\},&\text{L9a48}\{1,0\}, &\text{L9a48}\{1,0\}. &&
\end{array}
\]
The following table of pairwise linking numbers (for some ordering of the link components) shows, using Lemma~\ref{lem:linkingnumbers}, that the links in the table are not weakly doubly slice.

\begin{center}
\begin{tabular}{|c|c|c|c|}
\hline
Link & $\ell k(L_1,L_2)$ & $\ell k(L_1,L_3)$ & $\ell k(L_2,L_3)$ \\
\hline
\hline
\text{L8a19}\{0,0\} &-1&1&0\\
\text{L8a19}\{1,1\}&1&-1&0\\
\text{L8n3}\{1,0\}&1&-1&2\\
\text{L8n3}\{0,1\}&-1&1&2\\
\text{L9a46}\{0,0\}&0&0&-1\\
\text{L9a46}\{1,1\}&0&0&-1\\
\text{L9a48}\{1,0\}&1&-1&-2\\
\text{L9a48}\{0,1\}&1&-1&-2\\
\hline
\end{tabular}
\end{center}
The remaining links to rule out are the three quasi-orientation possibilities for L9a45. Let $L$ denote L9a45 with any of the orientations listed. We compute that $H_1(\Sigma_2(L))\cong\Z/2\oplus \Z/18$ and hence by Proposition~\ref{prop:doublebranched}, these links are not weakly doubly slice.
\end{proof}

We now treat the case of the Borromean rings. This link resists all our abelian invariants as the next example shows.
  \begin{example}[Borromean rings] \label{ex:borromean}
For the Borromean rings~$L=L6a4$, and for any orientation, some calculations show that~$\sigma_L^{LT}\equiv 0$, $H_1(X_L;\Q[t^{\pm 1},(1-t)^{-1}])=0$ and~$H_1(\Sigma_2(L))=\Z_4 \oplus \Z_4$; additionally the linking matrix of~$L$ is identically zero.
  \end{example}
  
However, the Borromean rings are \emph{not} weakly doubly slice. This will follow from the much stronger result below, which is known to experts but does not appear in print as far as we can tell. We are grateful to Peter Teichner for the argument in the proof of the following proposition.

\begin{proposition}\label{prop:borromean} The Borromean rings (with any orientation) do not bound a properly embedded connected genus 0 surface in~$D^4$.
\end{proposition}

The argument requires us to recall some more tools from $4$-manifold topology and in particular will use the~$\tau$ invariant of an immersed $2$-sphere in a $4$-manifold. We refer the reader to~\cite{SchneidermanTeichner} for a detailed account of the $\tau$ invariant, which we only briefly discuss now.

A generically immersed sphere~$S$ in a topological~$4$-manifold~$W$ is said to be \emph{$s$-characteristic} if~$S\cdot S'\equiv S'\cdot S'\pmod{2}$ for every generically immersed sphere~$S'$ in~$W$, where~$\cdot$ denotes the algebraic intersection number. Now suppose~$W$ is simply connected and that~$S$ is a generically immersed~$s$-characteristic~$2$-sphere in~$M$ such that the signed count of the self-intersections of $S$ is 0.
In this situation the following definition of the~$\tau$ invariant may be used~\cite[Remark~5]{SchneidermanTeichner}. Choose a set~$\{W_i\}$ of framed, generically immersed Whitney discs in~$M$, pairing the self-intersection points of~$S$, and such that each~$W_i$ intersects~$S$ transversely in double points in the interior of~$\{\mathring{W_i}\}$. Given such a set, define a signed intersection count
\[
\tau(S,\{W_i\}):=\sum_i S\cdot \mathring{W_i} \pmod{2}.
\]
The value of~$\tau(S,\{W_i\})\in \Z/2$ does not depend on choices of pairing of double points, Whitney arcs, or Whitney discs~\cite[Theorem~1]{SchneidermanTeichner}. Thus define~$\tau(S):=\tau(S,\{W_i\})$ for any such choice. 
Note that if $S$ is embedded then clearly $\tau(S)=0$ as we choose the Whitney disc collection to be empty. We recall as well that $\tau(S)$ is an invariant of the homotopy class $[S]\in\pi_2(M)$ of a generically immersed $2$-sphere (computable from homotopy representatives with vanishing self-intersection); see e.g.~\cite[\textsection 6.1]{FMNOPR} for a discussion. 

\begin{proof}[Proof of Proposition \ref{prop:borromean}]  As $\pi_2(S^4)=0$, and this unique homotopy class is represented by an \emph{embedded} $2$-sphere (for example the unknotted $S^2$), this implies that every generically immersed $2$-sphere in $S^4$ with vanishing self-intersection has vanishing $\tau$ invariant.

Now suppose~$F\subset D^4$ is a properly embedded connected genus 0 surface bounding the Borromean rings~$L\subset S^3$. Consider~\cite[Figure 1]{ConantSchneidermanTeichner}. This depicts the Borromean rings~$L\subset S^3$ together with a specific choice of three embedded discs~$D_1,D_2,D_3$ in~$D^4$, bounded by~$L$, that intersect each other exactly twice in~$D_1\cap D_2=\{q,q'\}$, and with opposite orientations. Define a generically immersed $2$-sphere $S:=F\cup D_1\cup D_2\cup D_3\subset S^4$, and note it is $s$-characteristic simply because $S^4$ has vanishing intersection form. In~\cite[Figure 1]{ConantSchneidermanTeichner}, there is furthermore depicted a single framed, embedded Whitney disc~$W$ pairing the intersection points~$q$ and~$q'$, and meeting~$D_3$ exactly once in some point~$p$. Thus we calculate that~$\tau(S)=1$. This contradicts our previous reasoning that $\tau(S)=0$, and so implies~$F$ does not exist.
\end{proof}

\begin{corollary} The Borromean rings (with any orientation) are not weakly doubly slice.
\end{corollary}

\begin{proof} Write $L$ for the Borromean rings with some choice of orientation and suppose (for a contradiction) $L$ is weakly doubly slice. Then $L$ is the cross section of a~$2$-sphere meeting an equatorial~$S^3$ in one of the configurations depicted in Figure~\ref{fig:wiggle}. 
The figure on the left cannot occur: if it did, then~$L$ would be slice, which is impossible since~$\mu_{123}(L)=1$ and Milnor's invariants are invariant under link concordance~\cite{Stallings,Casson}.
For the figure on the right, tube the annulus of~$F$ to the disc of~$F$ to see a connected genus $0$ surface, contradicting Proposition~\ref{prop:borromean}.
\end{proof}

\subsection{Orientations and double sliceness}

In Section \ref{sec:weakly} we saw that the unoriented links L6n1 and L8n8 are not weakly doubly slice with one quasi-orientation but are weakly doubly slice with another.
McCoy and McDonald have asked about the existence of such examples among~$2$-component links~\cite[Question 3]{McCoyMcDonald}.
Our multivariable invariants are fairly sensitive to orientation changes and we use this to provide such an example.

\begin{example}
\label{ex:Orientations}
Set~$K:=8_{20}$, and consider the unoriented link~$L=K_1 \cup K_2$ obtained as the~$(2,0)$-cable of~$K$.
We argue that~$L$ is weakly doubly slice with respect to one of its quasi-orientations but not with respect to the other.
If we orient~$K_1$ and~$K_2$ in opposite directions, then the fact that~$K$ is slice implies~$L$ is weakly doubly slice~\cite[Proposition 6]{McDonald}.
Next, we orient~$L$ so that~$K_1$ and~$K_2$ (which are copies of~$K$) have ``parallel" orientations and we use Corollary~\ref{cor:WeaklyDoublySlice} to show that~$L$ is not weakly doubly slice:
Set~$\omega:=e^{\pi i/6}$.
The Levine-Tristram signature of~$K$ vanishes identically, except at~$\omega^2$ and~$\overline{\omega}^2$, where~$\sigma_K(\omega^2)=\sigma_K(\overline{\omega^2}) = 1$.
Since with these orientations,~$L$ is obtained as the satellite on a winding number 2 pattern, we deduce from the cabling formula for the Levine-Tristram signature~\cite[item (2) of the Remark on page 76]{LitherlandIterated} that~$ \sigma_{L}^{LT}(\omega)=\sigma_K(\omega^2)=1.$
Corollary~\ref{cor:WeaklyDoublySlice} now implies that~$L$ is not weakly doubly slice.
Finally, note that Theorem~\ref{thm:MainTheoremIntro} shows that the ordered link $L$ is not strongly doubly slice since its multivariable signature is not identically zero; this also follows since the components of $L$ are not doubly slice.
\end{example}

\bibliographystyle{annotate}
\bibliography{BiblioDoublySlice}

\end{document}